\documentclass{gtpart}     % Basic GT/GTM/AGT style
\usepackage{faktor}
\usepackage[all,cmtip]{xy}
\usepackage{amscd}
\usepackage{mathrsfs}
\usepackage{color}
\usepackage{mathtools}
\usepackage[mathcal]{euscript}
\usepackage{hyperref}

\title{Betti numbers and stability for configuration spaces\\via factorization homology}

\author{Ben Knudsen}
\givenname{Ben}
\surname{Knudsen}
\address{Department of Mathematics\\
Harvard University\\
1 Oxford St\\
Cambridge, MA 02138}
\email{knudsen@math.harvard.edu}
\urladdr{scholar.harvard.edu/knudsen/}

\theoremstyle{definition}\newtheorem{definition}{Definition}[section]
\theoremstyle{theorem}
\theoremstyle{remark}\newtheorem*{conventions}{Conventions}
\theoremstyle{remark}\newtheorem*{acknowledgments}{Acknowledgments}
\theoremstyle{remark}\newtheorem{example}[definition]{Example}
\theoremstyle{definition}\newtheorem*{convention}{Convention}
\theoremstyle{theorem}\newtheorem{theorem}[definition]{Theorem}
\theoremstyle{remark}\newtheorem{remark}[definition]{Remark}
\theoremstyle{corollary}\newtheorem{corollary}[definition]{Corollary}
\theoremstyle{theorem}\newtheorem{proposition}[definition]{Proposition}
\newcommand{\Conf}{\mathrm{Conf}}
\newcommand{\Ch}{\mathcal{C}\mathrm{h}}
\newcommand{\Sym}{\mathrm{Sym}}

\newcommand{\Emb}{\mathrm{Emb}}

\newcommand{\Mfld}{\mathcal{M}\mathrm{fld}}
\newcommand{\hofiber}{\mathrm{hofiber}}
\newcommand{\Alg}{\mathrm{Alg}}
\newcommand{\Mod}{\mathrm{Mod}}
\newcommand{\Lie}{\mathcal{L}}
\newcommand{\Disk}{\mathcal{D}\mathrm{isk}}
\newcommand{\Fr}{\mathrm{Fr}}
\newcommand{\fr}{\mathrm{fr}}
\newcommand{\sgn}{\mathrm{sgn}}
\newcommand{\Map}{\mathrm{Map}}

\DeclareMathOperator*{\hocolim}{\mathrm{hocolim}}
\DeclareMathOperator*{\colim}{\mathrm{colim}}

\begin{document}

\begin{abstract}    % type your abstract below
Using factorization homology, we realize the rational homology of the unordered configuration spaces of an arbitrary manifold $M$, possibly with boundary, as the homology of a Lie algebra constructed from the compactly supported cohomology of $M$. By locating the homology of each configuration space within the Chevalley-Eilenberg complex of this Lie algebra, we extend theorems of B\"{o}digheimer-Cohen-Taylor and F\'{e}lix-Thomas and give a new, combinatorial proof of the homological stability results of Church and Randal-Williams. Our method lends itself to explicit calculations, examples of which we include.
\end{abstract}

\maketitle

%%%%%%%%%%%%%%%%%%%%   Start of main body of article

\section{Introduction}

We study the configuration space $B_k(M)$ of $k$ unordered points in a manifold $M$, defined as $$B_k(M)=\Conf_k(M)_{\Sigma_k}:=\big\{(x_1,\ldots,x_k)\in M^k:x_i\neq x_j \text{ for } i\neq j\big\}/\Sigma_k,$$ where $\Sigma_k$ acts by permuting the $x_i$. Our main theorem concerns the homology of these spaces.

\begin{theorem}\label{with grading} Let $M$ be an $n$-manifold.  There is an isomorphism of bigraded vector spaces $$\bigoplus_{k\geq0}H_*(B_k(M);\mathbb{Q})\cong H^\Lie(H_{c}^{-*}(M;\Lie(\mathbb{Q}^w[n-1]))).$$ Here $H_c^{-*}$ denotes compactly supported cohomology, $\mathbb{Q}^w$ is the orientation sheaf of $M$, $H^\Lie$ denotes Lie algebra homology, and $\Lie$ is the free graded Lie algebra functor. The auxiliary grading on the left is by cardinality of configuration and on the right by powers of the Lie generator.
\end{theorem}

Our methods apply equally to the calculation of the twisted homology of configuration spaces and of the homology of certain relative configuration spaces defined for manifolds with boundary; precise statements may be found in Theorem \ref{twisted version} and Theorem \ref{boundary case}, respectively. All results and arguments herein are valid over an arbitrary field of characteristic zero.

The study of configuration spaces is classical. To name some highlights, the space $B_k(\mathbb{R}^2)$ is a classifying space for the braid group on $k$ strands (see \cite{Artin}); the space $\Conf_k(\mathbb{R}^n)$ has the homotopy type of the space of $k$-ary operations of the little $n$-cubes operad and so plays a central role in the theory of $n$-fold loop spaces (see e.g. \cite{CLM}, \cite{May}, \cite{Segal}); certain spaces of labeled configurations provide models for more general types of mapping spaces (see \cite{Bodigheimer}, \cite{McDuff}, \cite{Salvatore}, \cite{Segal}); and, according to a striking theorem of \cite{LS}, the homotopy type of $B_k(M)$ is not an invariant of the homotopy type of $M$. 

As this last fact indicates, configuration spaces depend in subtle ways on the structure of the background manifold. On the other hand, the homology of these spaces has often been shown to be surprisingly simple, provided one is willing to work over a field of characteristic zero. Indeed, B\"{o}digheimer-Cohen-Taylor show that the Betti numbers of $B_k(M)$ are determined by those of $M$ when $M$ is of odd dimension, and F\'{e}lix-Thomas show that, in the even-dimensional case, the Betti numbers of $B_k(M)$ are determined by the rational cohomology ring of $M$, as long as $M$ is closed, orientable, and nilpotent; see \cite{BCT} and \cite{FT}, respectively. We recover extensions of these results as immediate consequences of Theorem \ref{with grading}.

\begin{corollary}\label{main corollary} The groups ${H}_*(B_k(M);\mathbb{Q})$ depend only on $n$ and
\begin{itemize}
\item the graded abelian group $H_*(M;\mathbb{Q})$ if $n$ is odd; or
\item the cup product $H_{c}^{-*}(M;\mathbb{Q}^w)^{\otimes 2}\to H_c^{-*}(M;\mathbb{Q})$ if $n$ is even.
\end{itemize}
\end{corollary}

The computational power of Theorem \ref{with grading} lies in the bigrading, which permits one to isolate the homology of a single configuration space within the Chevalley-Eilenberg complex computing the appropriate Lie homology. Employing this strategy, we show that the chain complexes computing $H_*(B_k(M);\mathbb{Q})$ exhibited in \cite{BCT} and \cite{FT} are isomorphic to subcomplexes of the Chevalley-Eilenberg complex; precise statements appear in \S\ref{formulas}. Better yet, in dealing with the entire Chevalley-Eilenberg complex at once, one is able to perform computations for all $k$ simultaneously; see \S\ref{examples}. 

Another important aspect of the study of configuration spaces is the phenomenon of \emph{homological stability}. As $k$ tends to infinity, the Betti numbers of $B_k(M)$ are eventually constant, despite the absence of a map of spaces $B_k(M)\to B_{k+1}(M)$ in general (see \cite{Church}, \cite{CEF}, \cite{ORWstab}, \cite{CP}, and the references therein). Here too, characteristic zero is special. 

Regarding stability, we prove the following:

\begin{theorem}\label{stability theorem}
Let $M$ be a connected $n$-manifold with $n>1$. Cap product with the unit in $H^0(M;\mathbb{Q})$ induces a map $$H_*(B_{k+1}(M);\mathbb{Q})\to H_*(B_k(M);\mathbb{Q})$$ that is \begin{itemize}
\item an isomorphism for $*<k$ and a surjection for $*=k$ when $M$ is an orientable surface; and
\item an isomorphism for $*\leq k$ and a surjection for $*=k+1$ in all other cases.
\end{itemize} 
\end{theorem}

The sense in which the homology of configuration spaces forms a coalgebra, so that the cap product is defined, will be explained in \S\ref{comultiplicative structure}. We lack a conceptual explanation for the exceptional behavior in dimension 2, as it emerges from our argument solely as a numerical/combinatorial coincidence.

This result improves on the stable range of \cite{Church} and very slightly on that of \cite{ORWstab}. As in the former work, our stable range can be further improved if the low degree Betti numbers of $M$ vanish. As the example of the Klein bottle shows, the bound $*\leq k$ is sharp in the sense that no better stable range holds for all manifolds that are not orientable surfaces. When $M$ is open, the surjectivity statement is proven in \cite{ORWstab}; to the author's knowledge, the result is new for compact manifolds.

Conceptually, we think of Theorem \ref{with grading} as providing an explanation and organizing principle for the behavior of configuration spaces in characteristic zero. The germ of our approach, and the source of the connection to Lie algebras, is the calculation, due to Arnold and Cohen, of the homology of the ordered configuration spaces of $\mathbb{R}^n$, which is the fundamental result of the subject (see \cite{Arnold}, \cite{CLM}). Specifically, for $n\geq2$, the homology groups of the spaces $\Conf_k(\mathbb{R}^n)$ form a shifted version of the operad governing Poisson algebras, with the shifted Lie bracket given by the fundamental class of $\Conf_2(\mathbb{R}^n)\simeq S^{n-1}$ (see \cite{Sinha} for a beautiful geometric discussion of this identification). Locally, then, configuration spaces enjoy a rich algebraic structure; factorization homology, our primary tool in this work, provides a means of assembling this structure across coordinate patches of a general manifold, globalizing the calculation of Arnold and Cohen. Theorem \ref{with grading} is the natural output of this procedure.

At a more formal level, we rely on the fact that the factorization homology of $M$, with coefficients taken in a certain free algebra, can be computed in two different ways. On the one hand, according to Proposition \ref{free calculation}, it has an expression in terms of the configuration spaces of $M$. On the other hand, the free algebra may be thought of as a kind of enveloping algebra, and a calculation of \cite{Knudsen} identifies the same invariant as Lie algebra homology. On the face of it, these calculations only coincide for framed manifolds; we show that they agree in general in characteristic zero.

In keeping with our metamathematical goal of making the case for factorization homology as a computational tool, we do not focus on the technical underpinnings of the theory. The interested reader may find these in \cite{Francis}, \cite{AFZ}, \cite{AF}, \cite{AFT1}, \cite{AFT2} \cite{FrancisTangent}, and \cite{HA}.

The paper following the introduction is split into seven sections. In \S2-3, we review the basics of factorization homology and discuss calculations thereof in several cases of interest. Theorem \ref{with grading} and its variants are proved in \S4 assuming several deferred results, and the classical results alluded to above follow. In \S5, we discuss coalgebraic phenomena arising from configuration spaces, which lead us to the proof of Theorem \ref{stability theorem} and one of the missing ingredients in the main theorem. Finally, \S6 is concerned with explicit computations, and \S7 supplies the remaining missing ingredients. 

\begin{conventions}\begin{enumerate}
\item In accordance with the bulk of the literature on factorization homology, we work in an $\infty$-categorical context, where for us an $\infty$-category will always mean a quasicategory. The standard references here are \cite{HTT} and \cite{HA}, but we will need to ask only very little of the vast theory developed therein, and the reader may obtain a sense of the arguments and results by substituting ``homotopy colimit'' for ``colimit'' everywhere, for example.
\item Every manifold is smooth and may be embedded as the interior of a compact manifold with boundary (such an embedding is not part of the data). We view manifolds as objects of the $\infty$-category $\Mfld_n$, the topological nerve of the topological category of $n$-manifolds and smooth embeddings, which is symmetric monoidal under disjoint union.
\item Our homology theories are valued in $\Ch_\mathbb{Q}$, the underlying $\infty$-category of the category of $\mathbb{Q}$-chain complexes equipped with the standard model structure. With the single exception of Theorem \ref{Eilenberg-Steenrod}, $\Ch_\mathbb{Q}$ is understood to be symmetric monoidal under tensor product. 
\item The homology of a chain complex $V$ is written $H(V)$, while the homology of a space $X$ is written $H_*(X)$. Hence $H_*(X)=H(C_*(X))$. If $\mathfrak{g}$ is a differential graded Lie algebra, then $H(\mathfrak{g})$ is a graded Lie algebra.
\item Chain complexes are homologically graded. If $V$ is a chain complex, $V[k]$ is the chain complex with $(V[k])_n=V_{n-k},$ and, for $x\in V$, the corresponding element in $V[k]$ is denoted $\sigma^kx$. Cohomology is concentrated in negative degrees; to reinforce this point, we write $H^{-*}(X)$ for the graded vector space whose degree $-k$ part is the $k$th cohomology group of $X$; for example, $$H^{-*}(S^n;\mathbb{Q})\cong\begin{cases}
\mathbb{Q}\quad&*\in\{-n,0\}\\
0\quad&\text{else.}
\end{cases}$$
\item If $X$ is a space and $V$ is a chain complex, the \emph{tensor of $X$ with $V$} is the chain complex $$X\otimes V:=C_*(X)\otimes V.$$
\item If $(X,A)$ is a pair of spaces, the quotient of $X$ by $A$ is the pointed space $X/A$ defined as the pushout in the diagram $$\xymatrix{
A\ar[d]\ar[r]&X\ar[d]\\
\mathrm{pt}\ar[r]&X/A.
}$$ In particular, we have $X/\varnothing=X_+$.
\item If $X$ is an object of the $\infty$-category $\mathcal{C}$ with an action of the group $G$, then $X_G$ denotes the $G$-coinvariants of $X$ (resp. $X^G$, invariants), which is an object of $\mathcal{C}$. When $\mathcal{C}$ is topological spaces or chain complexes, this object coincides in the homotopy category with \emph{homotopy} coinvariants.
\end{enumerate}
\end{conventions}

\begin{acknowledgments}
I am grateful to John Francis for suggesting this project, and for his help and patience. 

I would like to thank Lee Cohn, Elden Elmanto, Boris Hanin, Sander Kupers, Jeremy Miller, Martin Palmer, and Dylan Wilson for their comments on earlier versions of this paper; Joel Specter and Eric Wofsey for suggested examples; Bruno Vallette for an enlightening lecture; and Jordan Ellenberg and Aaron Mazel-Gee for crucial conversations leading to significant improvements of the results. 

After writing the paper, I learned that some of the same results are accessible through the work of Ezra Getzler by first passing to symmetric group invariants in the complex constructed in \cite{Getzler2step} and \cite{GetzlerHodge}, and then invoking Proposition \ref{technical lemma} below. I am grateful to Dan Petersen for bringing these papers to my attention. 

Revision was undertaken during visits to the Mathematisches Forschungsinstitut Oberwolfach and the Hausdorff Research Institute for Mathematics in Bonn.

Finally, I would like to thank the anonymous referees for their comments.
\end{acknowledgments}

\section{Factorization Homology}

\subsection{Homology theories}\label{factorization homology} In this section, we review the basic notions of factorization homology, alias topological chiral homology. The primary reference is \cite{Francis}. As there, our point of view is that factorization homology is a natural theory of homology for manifolds. To illustrate in what sense this is so, we first recall the classical characterization of ordinary homology, phrased in a way that invites generalization.

\begin{theorem}[Eilenberg-Steenrod]\label{Eilenberg-Steenrod}
Let $V$ be a chain complex. There is a symmetric monoidal functor $C_*(-; V)$ from spaces with disjoint union to chain complexes with direct sum, called \emph{singular homology with coefficients in $V$}, which is characterized up to natural equivalence by the following properties:
\begin{enumerate}
\item $C_*(\mathrm{pt}; V)\simeq V$;
\item the natural map $$\xymatrix{\displaystyle C_*(X_1;V)\bigoplus_{C_*(X_0;V)}C_*(X_2;V)\ar[r]& C_*(X;V)}$$ is an equivalence, where $X$ is the pushout of the diagram of cofibrations $X_1\hookleftarrow X_0\hookrightarrow X_2$.
\end{enumerate}
\end{theorem}

Property (ii), a local-to-global principle equivalent to the usual excision axiom, is the reason that homology is computable and hence useful. 

Of course, ordinary homology is a homotopy invariant. In the study of manifolds, the equivalence relation of interest is often finer than homotopy equivalence, and one could hope for a theory better suited to such geometric investigations. To discover what form this theory might take, let us contemplate a generic symmetric monoidal functor $(\Mfld_n,\amalg)\to (\Ch_\mathbb{Q},\otimes)$. By analogy with Theorem \ref{Eilenberg-Steenrod}, we ask that this functor be determined by its value on $\mathbb{R}^n$, the basic building block in the construction of $n$-manifolds. Unlike a point, however, Euclidean space has interesting internal structure.

\begin{definition}
An \emph{$n$-disk algebra} in $\Ch_\mathbb{Q}$ is a symmetric monoidal functor $$\xymatrix{A:(\Disk_n,\amalg)\ar[r]& (\Ch_\mathbb{Q},\otimes)}$$ where $\Disk_n\subseteq \Mfld_n$ is the full subcategory spanned by manifolds diffeomorphic to $\amalg_k\mathbb{R}^n$ for some $k\in\mathbb{Z}_{\geq0}$.
\end{definition}

In other words, $\Disk_n$ is the (nerve of the) category of operations associated to the endomorphism operad of the manifold $\mathbb{R}^n$, and an $n$-disk algebra is an algebra over this operad. In contrast, the endomorphism operad of a point in topological spaces is the commutative operad, and every chain complex is canonically and essentially uniquely a commutative algebra in $(\Ch_\mathbb{Q},\oplus)$.

Taking the extra structure of $\mathbb{R}^n$ into account, \cite[3.24]{Francis} provides an analogous classification theorem.

\begin{theorem}[Ayala-Francis]\label{homology theories}
Let $A$ be an $n$-disk algebra. There is a symmetric monoidal functor $\int_{(-)}A$ from $n$-manifolds with disjoint union to chain complexes with tensor product, called \emph{factorization homology with coefficients in $A$}, which is characterized up to natural equivalence by the following properties:
\begin{enumerate}
\item $\int_{\mathbb{R}^n}A\simeq A$ as $n$-disk algebras;
\item the natural map $$\xymatrix{\displaystyle\int_{M_1}A\bigotimes_{\displaystyle\int_{M_0\times\mathbb{R}}A}\int_{M_2}A\ar[r]& \displaystyle\int_{M}A}$$ is an equivalence, where $M$ is obtained as the collar-gluing of the diagram of embeddings $M_1\hookleftarrow M_0\times\mathbb{R}\hookrightarrow M_2.$
\end{enumerate}
\end{theorem}

Just as the functor of singular chains is but one model for ordinary homology, factorization homology may be constructed in several equivalent ways. The construction that we will favor is as follows.

Let $A:\Disk_n\to \Ch_\mathbb{Q}$ be an $n$-disk algebra. Then factorization homology with coefficients in $A$ is the left Kan extension in the following diagram of $\infty$-categories:
$$\xymatrix{
\Disk_n\ar[rr]^A\ar[dd] &&\Ch_\mathbb{Q}\\\\
\Mfld_n \ar@{-->}[uurr]_{\int_{(-)}A}
}$$ Explicitly, it may be calculated as the colimit $$\xymatrix{\displaystyle\int_MA\simeq \colim\bigg(\Disk_{n/M}\ar[r]&\Disk_n\ar[r]^-A&\displaystyle\Ch_\mathbb{Q}\bigg).}$$

\begin{remark}
Since $\Ch_\mathbb{Q}$ admits sifted colimits and $\otimes$ distributes over them, \cite[3.2.3]{AFZ} guarantees that the left Kan extension and the symmetric monoidal left Kan extension exist and coincide.
\end{remark}

\subsection{Variant: framed manifolds} 

The category $\Disk_n$ is closely related to the classical operad $E_n$ of little $n$-cubes. To make this connection, we recall that a \emph{framing} of an $n$-manifold $M$ is a nullhomotopy of its tangent classifier $$\xymatrix{M\ar[rr]^-{TM\simeq*}&&BO(n)}.$$ With the corresponding notion of \emph{framed embedding} between framed manifolds in hand, one obtains an $\infty$-category $\Mfld_n^\fr$ of framed $n$-manifolds (see \cite[2.7]{Francis}).

\begin{definition}
A \emph{framed $n$-disk algebra} in $\Ch_\mathbb{Q}$ is a symmetric monoidal functor $A:(\Disk_n^\fr,\amalg)\to (\Ch_\mathbb{Q},\otimes)$, where $\Disk_n^\fr\subseteq\Mfld_n^\fr$ is the full subcategory spanned by framed manifolds diffeomorphic to $\amalg_k\mathbb{R}^n$ for some $k\in\mathbb{Z}_{\geq0}$.
\end{definition}

As before, the factorization homology of a framed $n$-manifold with coefficients in a framed $n$-disk algebra is defined as the left Kan extension from $\Disk_n^\fr$. Indeed, the whole theory carries over into the context of topological manifolds equipped with a microtangential $B$-structure arising from a map $B\to B\mathrm{Top}(n)$. In this paper, we will only make use of the cases $B=BO(n)$, corresponding to smooth manifolds (see \cite[2.11, 3.29]{Francis}), and $B=*$, corresponding to framed manifolds.

Now, the topological operad $E_n$ has an associated $\infty$-operad (see \cite[2.1]{HA}), and \cite[2.11]{AFT2} asserts an equivalence $$\xymatrix{\Alg_{\Disk_n^\fr}(\mathcal{C})\ar[r]^-\sim&\Alg_{E_n}(\mathcal{C})}$$ for any symmetric monoidal $\infty$-category $\mathcal{C}$. Moreover, this equivalence induces a further equivalence $$\xymatrix{\Alg_{\Disk_n}(\mathcal{C})\ar[r]^-\sim&\Alg_{E_n}(\mathcal{C})^{O(n)}.}$$ Informally, an $n$-disk algebra is an $E_n$-algebra with an action of $O(n)$ compatible with the action on $E_n$ given by rotating disks. In the language of \cite{SalvatoreWahl}, $n$-disk algebras are algebras for the \emph{semidirect product} $E_n\rtimes O(n)$.

\begin{remark}
The reader is cautioned not to confuse the framed $n$-disk algebras employed here with the ``framed $E_n$-algebras'' that occur elsewhere in the literature. These algebras carry an action of $SO(n)$ and yield homology theories for \emph{oriented} manifolds.
\end{remark}

\subsection{Free algebras} We introduce several functors that will be important for us in what follows. The reference here is \cite{AF}.

Within the $\infty$-category $\Disk_n$ there is a Kan complex with a single vertex, the object $\mathbb{R}^n$, whose endomorphisms are $\Emb(\mathbb{R}^n,\mathbb{R}^n)\simeq O(n)$, so that we may identify this Kan complex with $BO(n)$. Restricting to this subcategory defines a forgetful functor $$\xymatrix{\Alg_{\Disk_n}(\Ch_\mathbb{Q})\ar[r]&\mathrm{Fun}(BO(n),\Ch_\mathbb{Q})\ar[r]^-\sim&\Mod_{O(n)}(\Ch_\mathbb{Q})}$$ The latter symbol denotes the $\infty$-category of chain complexes equipped with an action of $C_*(O(n);\mathbb{Q})$, which we refer to simply as $O(n)$-\emph{modules}. This functor admits a left adjoint $\mathbb{F}_n$, the free $n$-disk algebra generated by an $O(n)$-module. 

Evaluation on $\mathbb{R}^n$ defines a still more forgetful functor, which we think of as associating to an algebra its underlying chain complex. The situation is summarized in the following commuting diagram of adjunctions, in which the straight arrows are right and the bent arrows left adjoints:
$$\xymatrix{
\Alg_{\Disk_n}(\Ch_\mathbb{Q})\ar[rr]\ar[ddrr]&&\Mod_{O(n)}(\Ch_\mathbb{Q})\ar@/_1pc/[ll]_{\mathbb{F}_n}\ar[dd]\\\\
&&\Ch_\mathbb{Q}\ar@/_1pc/[uu]_{O(n)\otimes(-)}\ar@/^1pc/[uull]
}$$

In particular, for a chain complex $V$, the free $n$-disk algebra on $V$ is naturally equivalent to $\mathbb{F}_n(O(n)\otimes V).$ More generally, there is the following description.

\begin{proposition}\label{free identification}
There is a natural equivalence $$\xymatrix{\mathbb{F}_n(K)\ar[r]^-\sim&\displaystyle\bigoplus_{k\geq0}\bigg(\Emb(\amalg_k\mathbb{R}^n,-)\bigotimes_{\Sigma_k\ltimes O(n)^k}K^{\otimes k}\bigg),}$$ where $K$ is an $O(n)$-module.
\end{proposition}
\begin{proof}
The map is supplied by the universal property of the free algebra. In the case $K=O(n)\otimes V$, it is an equivalence, since $\mathbb{F}_n(K)$ is now the free $n$-disk algebra on the chain complex $V$, so that \begin{align*}\mathbb{F}_n(K)&\simeq\bigoplus_{k\geq0}\bigg(\Emb(\amalg_k\mathbb{R}^n,-)\bigotimes_{\Sigma_k}V^{\otimes k}\bigg)\\
&\cong\bigoplus_{k\geq0}\bigg(\Emb(\amalg_k\mathbb{R}^n,-)\bigotimes_{\Sigma_k\ltimes O(n)^k}(O(n)^{k}\otimes V^{\otimes k})\bigg)\\
&\cong\bigoplus_{k\geq0}\bigg(\Emb(\amalg_k\mathbb{R}^n,-)\bigotimes_{\Sigma_k\ltimes O(n)^k}K^{\otimes k}\bigg).
\end{align*} Since a general $K$ may be expressed as a split geometric realization of free $O(n)$-modules, and since $\mathbb{F}_n$, as a left adjoint, preserves geometric realizations, it suffices to show that the right hand side shares this property. But both $\Mod_{O(n)}(\Ch_\mathbb{Q})$ and $\Alg_{\Disk_n}(\Ch_\mathbb{Q})$ are monadic over $\Ch_\mathbb{Q}$, so on both sides the geometric realization is computed in $\Ch_\mathbb{Q}$, and the right hand side clearly preserves colimits in chain complexes.
\end{proof}

In the framed case, $\Emb^\fr(\mathbb{R}^n,\mathbb{R}^n)$ is contractible, so there is only the one forgetful functor $$\xymatrix{\Alg_{\Disk_n^\fr}(\Ch_\mathbb{Q})\ar[r]& \Ch_\mathbb{Q},
}$$ whose left adjoint, the free framed $n$-disk algebra functor, is denoted $\mathbb{F}_n^\fr$.

By restriction along the natural inclusion $\Disk_n^\fr\to \Disk_n$, any $n$-disk algebra is in particular a framed $n$-disk algebra, and there is an equivalence of $\Disk^\fr_n$-algebras $$\mathbb{F}_n(V)\simeq \mathbb{F}^\fr_n(V),$$ where $V$ is a chain complex considered as a trivial $O(n)$-module.

\section{Calculations}\label{calculations}

\subsection{Frame bundles}\label{free algebras}

The object of this section is twofold. First, we compute the factorization homology of the free $n$-disk algebra generated by an $O(n)$-module $K$. Second, for suitable $K$, we interpret this calculation in terms of the homology of configuration spaces.

For a manifold $M$, let $\Fr_M\to M$ denote the corresponding principal $O(n)$-bundle. Since $\Conf_k(M)$ is an open submanifold of $M^k$, its structure group is canonically reducible to $O(n)^k$, and we denote the corresponding principal $O(n)^k$-bundle by $\Conf^{\mathrm{fr}}_k(M)$.

\begin{proposition}\label{free calculation}
There is a natural equivalence $$\xymatrix{\displaystyle \int_M\mathbb{F}_n(K)\ar[r]^-\sim&\displaystyle \bigoplus_{k\geq0}\bigg(\Conf^\mathrm{fr}_k(M)\bigotimes_{\Sigma_k\ltimes O(n)^k}K^{\otimes k}\bigg),}$$ where $K$ is an $O(n)$-module.
\end{proposition}

\begin{proof}
The natural map $$\xymatrix{\displaystyle\colim_{\Disk_{n/M}}\bigg(\Emb(\amalg_k\mathbb{R}^n,-)\bigg)\ar[r]^-\sim&\Emb(\amalg_k\mathbb{R}^n,M)}$$ is an equivalence by \cite[p. 726]{HA}, so we have \begin{align*}\int_M\mathbb{F}_n(K)&\simeq\colim_{\Disk_{n/M}}\bigg(\bigoplus_{k\geq0}\bigg(\Emb(\amalg_k\mathbb{R}^n,-)\bigotimes_{\Sigma_k\ltimes O(n)^k}K^{\otimes k}\bigg)\bigg)\\
&\simeq \bigoplus_{k\geq0}\bigg(\colim_{\Disk_{n/M}}\bigg(\Emb(\amalg_k\mathbb{R}^n,-)\bigg)\bigotimes_{\Sigma_k\ltimes O(n)^k}K^{\otimes k}\bigg)\\
&\simeq\bigoplus_{k\geq0}\bigg(\Emb(\amalg_k\mathbb{R}^n,M)\bigotimes_{\Sigma_k\ltimes O(n)^k}K^{\otimes k}\bigg).
\end{align*} 

\noindent To conclude, we note that evaluation at the origin defines a projection $\Emb(\amalg_k\mathbb{R}^n,M)\to \Conf_k(M),$ and the natural map derivative map $\Emb(\amalg_k\mathbb{R}^n,M)\to\Conf^\mathrm{fr}_k(M)$ covering the identity is an equivalence of $O(n)^k$-spaces over $\Conf_k(M)$.
\end{proof}

\begin{remark}
This proposition is a special case of a calculation carried out in the more general context of \emph{zero-pointed manifolds} in \cite[2.4.1]{AF}. We have included this simplified argument for the reader's convenience.
\end{remark}

It will be important in what follows to be able to identify the summand of this object corresponding to a particular choice of $k$.

\begin{definition}\label{cardinality grading}
The \emph{cardinality grading} of the functor $\int_M\mathbb{F}_n(K)$ is the grading corresponding to the direct sum decomposition of Proposition \ref{free calculation}.
\end{definition}

Note that this grading corresponds to the grading induced on the colimit by the cardinality grading of the functor $\mathbb{F}_n(K)$.

We will be most interested in this calculation for particularly simple choices of $O(n)$-module $K$.

\begin{corollary}\label{homology}
There is a natural equivalence $$\xymatrix{\displaystyle\int_M\mathbb{F}_n(\mathbb{Q})\ar[r]^-\sim& \displaystyle\bigoplus_{k\geq0}C_*(B_k(M);\mathbb{Q}).}$$
\end{corollary}

Proposition \ref{free calculation} can also be used to study the twisted homology of $B_k(M)$. To pursue this direction, we must first identify the orientation cover $\widetilde{B_k(M)}$ of $B_k(M)$. For this we note that the orientation cover $$\widetilde{\Conf_k(M)}\to \Conf_k(M)\to B_k(M)$$ has structure group $\Sigma_k\times C_2$ when considered as a bundle over $B_k(M)$; that the automorphism corresponding to $-1\in C_2$ reverses orientation; and that the automorphism corresponding to $\tau\in \Sigma_k$ reverses orientation if $\sgn(\tau)=-1$ and $n$ is odd and preserves orientation otherwise. Therefore, the action of the subgroup $$H:=\{(\tau,\sgn(\tau)^n)\mid\tau\in \Sigma_k\}< \Sigma_k\times C_2$$ is orientation preserving, and we deduce the following:

\begin{proposition} $\widetilde{B_k(M)}\cong \widetilde{\Conf_k(M)}_H$ as covers of $B_k(M)$.
\end{proposition}

For a chain complex $V$, let $V^\mathrm{sgn}$ denote the sign representation of $C_2$ on $V$ and $V^{\det}$ the $O(n)$-module obtained from the latter by restriction along the determinant $O(n)\to C_2$. Recall that, for an $n$-manifold $N$, the homology of $N$ twisted by the orientation character may be computed as the homology of the complex $$C_*(N;\mathbb{Q}^w):=\widetilde N\bigotimes_{C_2}\mathbb{Q}^\mathrm{sgn}\cong \Fr_N\bigotimes_{O(n)}\mathbb{Q}^{\det}.$$

\begin{proposition}\label{twisted homology} Let $M$ be an $n$-manifold.
\begin{enumerate}
\item If $n$ is even, there is a natural equivalence $$\xymatrix{\displaystyle\int_M\mathbb{F}_n(\mathbb{Q}^{\det})\ar[r]^-\sim&\displaystyle \bigoplus_{k\geq0}C_*(B_k(M);\mathbb{Q}^w).}$$
\item If $n$ is odd, there is a natural equivalence$$\xymatrix{\displaystyle\int_M\mathbb{F}_n(\mathbb{Q}^{\det}[1])\ar[r]^-\sim&\displaystyle \bigoplus_{k\geq0}C_*(B_k(M);\mathbb{Q}^w)[k].}$$ 
\end{enumerate}
\end{proposition}
\begin{proof} \begin{enumerate}
\item
We have that \begin{align*}
\Conf^\mathrm{fr}_k(M)\bigotimes_{\Sigma_k\ltimes O(n)^k}(\mathbb{Q}^{\det})^{\otimes k}&\cong \Conf^\mathrm{fr}_k(M)\bigotimes_{\Sigma_k\ltimes O(nk)}\mathbb{Q}^{\det}
\\
&\cong\widetilde{\Conf_k(M)}\bigotimes_{\Sigma_k\times C_2}\mathbb{Q}^\mathrm{sgn}\\
&\cong\widetilde{\Conf_k(M)}_{\Sigma_k}\bigotimes_{C_2}\mathbb{Q}^\mathrm{sgn}\\ 
&\cong\widetilde{B_k(M)}\bigotimes_{C_2}\mathbb{Q}^\mathrm{sgn}, 
\end{align*} where we have used the commuting of the diagram $$\xymatrix{
O(n)^k\ar[r]\ar[d]_{\det^k} &O(nk)\ar[d]^{\det}\\
C_2^k\ar[r]^{\text{multiply}} &C_2.
}$$ and the fact that $H=\Sigma_k\times \{1\}$ when $n$ is even. The claim follows after summing over $k$ and applying Proposition \ref{free calculation}.
\item Similarly, we have that \begin{align*}
\Conf^\mathrm{fr}_k(M)\bigotimes_{\Sigma_k\ltimes O(n)^k}(\mathbb{Q}^{\det}[1])^{\otimes k}&\cong \Conf^\mathrm{fr}_k(M)\bigotimes_{\Sigma_k\ltimes O(nk)}(\mathbb{Q}^{\det}\otimes\mathbb{Q}[1]^{\otimes k})
\\
&\cong\widetilde{\Conf_k(M)}\bigotimes_{\Sigma_k\times C_2}(\mathbb{Q}^\mathrm{sgn}\otimes\mathbb{Q}[1]^{\otimes k})\\
&\cong\widetilde{\Conf_k(M)}_{H}\bigotimes_{C_2}\mathbb{Q}^\mathrm{sgn}[k]\\ 
&\cong\widetilde{B_k(M)}\bigotimes_{C_2}\mathbb{Q}^\mathrm{sgn}[k], 
\end{align*} where we have used that $\mathbb{Q}^\mathrm{sgn}\otimes\mathbb{Q}[1]^{\otimes k}$ is a trivial $H$-module and that $[\Sigma_k\times C_2:H]=2$.
\end{enumerate}
\end{proof}

\subsection{Commutative algebras}\label{commutative algebras}

We now consider a calculation of factorization homology in a certain degenerate case, which is a slight generalization of that considered in \cite[5.1]{Francis}. We will make use of this calculation in the next section.

Restriction of embeddings defines a map $\Emb(\amalg_k\mathbb{R}^n,\mathbb{R}^n)\to \prod_k\Emb(\mathbb{R}^n,\mathbb{R}^n)\simeq\prod_kO(n)$, which assemble to form a symmetric monoidal functor $$\xymatrix{\pi:\Disk_n\ar[r]&BO(n)^\amalg,}$$ where $BO(n)^\amalg$ is the $\infty$-category obtained as the nerve of the topological category with objects natural numbers and morphism spaces given by $$\mathrm{Map}_{BO(n)^\amalg}(r,s)=\coprod_{f:\langle r\rangle\to \langle s\rangle}\prod_{i=1}^sO(n)^{f^{-1}(i)},$$ which is symmetric monoidal under addition. For more on this and related $\infty$-categories, the reader may consult \cite[2.4.3]{HA}. For us, the relevance of this object is the following consequence of \cite[2.4.3.18]{HA}:

\begin{theorem}[(Lurie)]
There is an equivalence $$\xymatrix{
\mathrm{Fun}^\otimes(BO(n)^\amalg,\Ch_{\mathbb{Q}})\ar[r]^-\sim&\Mod_{O(n)}(\Alg_{\mathrm{Com}}(\Ch_\mathbb{Q})).
}$$
\end{theorem}

This result motivates our next definition.

\begin{definition}
A \emph{commutative refinement} of an $n$-disk algebra $A$ is a factorization $$\xymatrix{
\Disk_n\ar[d]_-\pi\ar[r]^-A&\Ch_\mathbb{Q}\\
BO(n)^\amalg\ar@{-->}[ur]
}$$ through a symmetric monoidal functor $BO(n)^\amalg\to\Ch_\mathbb{Q}$.
\end{definition}

By the previous theorem, a commutative refinement endows the underlying object of $A$ with the structure of a commutative algebra for which the $n$-disk algebra structure maps are homomorphisms. More formally, we obtain a factorization $$\xymatrix{
\Disk_n\ar[d]_-{A_\mathrm{Com}}\ar[r]^-A&\Ch_\mathbb{Q}\\
\Alg_\mathrm{Com}(\Ch_\mathbb{Q})\ar[ur]
}$$ of $A$ through the forgetful functor.

\begin{example}
By the K\"{u}nneth theorem, the functor $H:\Ch_\mathbb{Q}\to \Ch_\mathbb{Q}$ is symmetric monoidal, whence the homology of an $n$-disk algebra is canonically an $n$-disk algebra. Since $H$ factors through the discrete $\infty$-category of graded vector spaces, we have a symmetric monoidal factorization $$\xymatrix{
\Disk_n\ar@/_2.5pc/[dd]_-{\pi_0}\ar[d]_-\pi\ar[r]^-{H(A)}&\Ch_\mathbb{Q}\\
BO(n)^\amalg\ar[d]\\
BC_2^\amalg\ar@{-->}[uur]
}$$ through the homotopy category of $\Disk_n$, so that $H(A)$ is canonically commutative.
\end{example}

\begin{definition}
Let $X$ be a topological space and $B$ a commutative algebra. The \emph{tensor} of $X$ and $B$ is the colimit $$\xymatrix{\displaystyle X\otimes B=\colim\bigg(X\ar[r]&\mathrm{pt} \ar[r]^-B&\Alg_\mathrm{Com}(\Ch_\mathbb{Q})\displaystyle\bigg)}$$ of the constant functor from $X$, viewed as an $\infty$-groupoid, with value $B$.
\end{definition}

\begin{remark}
When $X=S^1$, this construction has the homotopy type of the Hochschild chains of $A$. In general, one recovers Pirashvili's higher Hochschild homology.
\end{remark}

Let $\Disk_{n/M}^{1}$ denote the full subcategory of $\Disk_{n/M}$ spanned by those arrows $\amalg_k\mathbb{R}^n\to M$ with $k=1$. 

\begin{proposition} Suppose that $A$ admits a commutative refinement. There is a natural equivalence 
$$\xymatrix{\displaystyle\Fr_M\bigotimes_{O(n)}A\simeq\colim\bigg(\Disk_{n/M}^1\ar[r]&\Disk_n\ar[r]^-{A_\mathrm{Com}}&\Alg_\mathrm{Com}(\Ch_\mathbb{Q})\displaystyle\bigg).}$$
\end{proposition}
\begin{proof}
Since the colimit is the left Kan extension to a point, and since Kan extensions compose, we may write $$\colim_{\Disk_{n/M}^1}A_{\mathrm{Com}}\simeq\colim_{BO(n)}\pi_!A_\mathrm{Com}\simeq (\pi_!A_\mathrm{Com})_{O(n)},$$ so that it suffices to identify $\pi_!A_\mathrm{Com}$. 

Since the projection $\Disk_{n/M}\to \Disk_n$ is a left fibration, so is $\pi:\Disk_{n/M}^1\to BO(n)$; in particular, this functor is a coCartesian fibration, which implies that the inclusion $\pi^{-1}(\mathrm{pt})\to \pi_{/\mathrm{pt}}$ of the fiber over the basepoint into the overcategory is a right adjoint and hence final. Therefore, we have $$\pi_!A_\mathrm{Com}=\colim_{\pi_{/\mathrm{pt}}}A_\mathrm{Com}\simeq \colim_{\pi^{-1}(\mathrm{pt})}A_\mathrm{Com}=\pi^{-1}(\mathrm{pt})\otimes A$$ According to \cite[2.13]{Francis}, the $\infty$-category $\Disk_{n/M}^1$ is equivalent to the Kan complex $M$, and the map $\pi:\Disk_{n/M}^1\to BO(n)$ coincides under this identification with the classifying map for the tangent bundle of $M$. In particular, the fiber of this map is $O(n)$-equivalent to $\Fr_M$, which completes the proof.
\end{proof}

\begin{proposition}\label{commutative calculation}
Suppose that $A$ admits a commutative refinement. There is a natural equivalence $$\int_MA\simeq\Fr_M\bigotimes_{O(n)}A.
$$
\end{proposition}
\begin{proof}
By the previous proposition, it suffices to show that the inclusion $\Disk_{n/M}^1\to \Disk_{n/M}$ and the forgetful functor $\Alg_\mathrm{Com}(\Ch_\mathbb{Q})\to \Ch_\mathbb{Q}$ induce equivalences $$\xymatrix{
\displaystyle\colim_{\Disk_{n/M}^1}A_\mathrm{Com}\ar[r]^-\sim&\displaystyle\colim_{\Disk_{n/M}}A_\mathrm{Com}\ar[r]^-\sim&\displaystyle\colim_{\Disk_{n/M}}A
}$$ when $A$ is commutative. 

Since $\Ch_\mathbb{Q}$ is $\otimes$-presentable (see \cite[3.4]{Francis}), the second equivalence follows from \cite[3.22]{Francis}, which asserts that $\Disk_{n/M}$ is sifted, and \cite[3.2.3.2]{HA}, which implies that the forgetful functor from commutative algebras preserves sifted colimits.

The first equivalence holds whenever $M$ is framed by \cite[5.1]{Francis}, since in this case the diagram $$\xymatrix{
\Disk^\fr_{n/M}\ar[d]\ar[r]^-\sim&\Disk_{n/M}\ar[d]\\
\Disk_n^\fr\ar[d]\ar[r]&\Disk_n\ar[d]\\
\mathrm{Fin}=\{e\}^\amalg\ar[r]&BO(n)^\amalg
}$$ commutes. In particular, the equivalence holds for $M=\amalg_k\mathbb{R}^n$, and we conclude that $$A\simeq\Fr_{(-)}\bigotimes_{O(n)}A$$ as $n$-disk algebras. Therefore, the claim will be established once we are assured that the expression on the right satisfies condition (ii) of Theorem \ref{homology theories}. For this, we note that the functor $\Fr_{(-)}$ takes a collar-gluing of manifolds to a pushout of $O(n)$-spaces, and that the functor $-\otimes_{O(n)}A$ preserves colimits of $O(n)$-spaces.
\end{proof}

\subsection{A spectral sequence}\label{spectral sequence section}

We employ a certain ``commutative-to-noncommutative'' spectral sequence in the proof of Theorem \ref{with grading}. For technical reasons, it will be convenient to restrict our attention to $n$-disk algebras valued in $\Ch_\mathbb{Q}^{\geq0}$, the full subcategory of chain complexes concentrated in non-negative homological degree. This restriction is not essential.

\begin{proposition}\label{spectral sequence} Let $M$ be an $n$-manifold and $A$ an $n$-disk algebra in $\Ch_\mathbb{Q}^{\geq0}$. There is a natural first-quadrant spectral sequence $$E^2_{p,q}\cong H_{p,q}\bigg(\Fr_M\bigotimes_{O(n)} H(A)\bigg)\implies H_{p+q}\bigg(\int_MA\bigg),$$ with differential $d^r$ of bidegree $(-r,r-1)$. 
\end{proposition}

The nature of the bigrading will become clear in the proof.

To construct this spectral sequence, we employ a rigidified version of the overcategory $\Disk_{n/M}$, denoted $\mathrm{Disj}(M)$ following \cite[5]{HA}, which is the poset of those open subsets of $M$ diffeomorphic to $\amalg_k\mathbb{R}^n$ for some $k$. We refer the reader to \cite[5.5.2.13]{HA} for the proof of the following result:

\begin{proposition}
There is a final functor $N(\mathrm{Disj}(M))\to \Disk_{n/M}$.
\end{proposition}

Thus, by \cite[4.1.1.8]{HTT}, the factorization homology of $M$ may be computed as a colimit over the nerve of the ordinary category $\mathrm{Disj}(M)$. Having achieved this simplification, we proceed as follows. Using the fact that $\Ch_\mathbb{Q}^{\geq0}$ arises from a combinatorial simplicial model category, \cite[4.2.4.4]{HTT} implies that any functor $N(\mathrm{Disj}(M))\to \Ch_\mathbb{Q}^{\geq0}$ of $\infty$-categories is equivalent in the $\infty$-category of functors to one coming from a functor of ordinary categories. Having chosen such a ``straightening'' of $A$, which we abusively denote by $A$, \cite[4.2.4.1]{HTT} now guarantees that the homotopy colimit of $A$ coincides with the $\infty$-categorical colimit. 

\begin{proof}[Proof of Proposition \ref{spectral sequence}]
From the discussion of the previous paragraph and \cite[5.1.3]{Riehl}, we have equivalences $$\int_MA\simeq \hocolim_{\mathrm{Disj}(M)}A\simeq B(\mathrm{pt},\mathrm{Disj}(M),A),$$ where $B(\mathrm{pt},\mathrm{Disj}(M),A)$ denotes the realization of the simplicial chain complex given in simplicial degree $p$ by $$B_p(\mathrm{pt},\mathrm{Disj}(M),A)=\bigoplus_{U_p\to\cdots\to U_0\to M}A(U_p)$$ (here we use for a second time the fact that the model structure on non-negatively graded chain complexes is simplicial).  Filtering by skeleta in the usual way, we obtain a spectral sequence $$E^1_{p,q}=\bigoplus_{U_p\to\cdots\to U_0\to M}H_q(A(U_p))\implies H_{p+q}\bigg(\int_M A\bigg),$$ with the differential $d^1$ given by the alternating sum of the face maps (see \cite[5.1]{SegalSS}, for example, which treats the case of a simplicial space). In other words, the $E^1$ page is the (graded) chain complex associated to the (graded) simplicial chain complex $B_\bullet(\mathrm{pt},\mathrm{Disj}(M),H(A))$ via the Dold-Kan correspondence, so that, invoking Proposition \ref{commutative calculation}, we have natural isomorphisms $$E^2_{p,q}\cong H_{p,q}(B(\mathrm{pt},\mathrm{Disj}(M),H(A))\cong H_{p,q}\bigg(\int_M H(A)\bigg)\cong H_{p,q}\bigg(\Fr_M\bigotimes_{O(n)} H(A)\bigg).$$
\end{proof}

\begin{remark}
Horel discusses a version of this spectral sequence in \cite[5]{Horel}.
\end{remark}

\subsection{Enveloping algebras}\label{Lie algebras}

In this section, we outline the place of Lie algebras in the theory of factorization homology, the general reference for which is \cite{Knudsen}.

It has long been known that configuration spaces are intimately related to Lie algebras (see \cite{CLM}, \cite{CohenTaylor}, or \cite{Cohen}, for example). To see the connection, suppose that $A$ is a $\Disk_n^\fr$-algebra in chain complexes, with $n\geq2$. Part of the structure of such an object is a multiplication map $$m:\Emb^\fr(\amalg_2\mathbb{R}^n,\mathbb{R}^n)\otimes A^{\otimes 2}\to A,$$ and since the homology of $\Emb^\fr(\amalg_2\mathbb{R}^n,\mathbb{R}^n)\simeq\Conf_2(\mathbb{R}^n)\simeq S^{n-1}$ is concentrated in degrees 0 and $n-1$, this multiplication encodes two maps $$m_0:A^{\otimes2}\to A\qquad\text{and}\qquad m_{n-1}: A^{\otimes 2}\to A[1-n]$$ defining a commutative multiplication on $A$ and a Lie bracket on $A[n-1]$, again up to homotopy. The Jacobi identity for $m_{n-1}$ follows from the ``three-term'' or ``Yang-Baxter'' relations in $H_*(\Conf_3(\mathbb{R}^n))$ (see \cite{FH}), and $O(n)$, acting on $S^{n-1}$ by degree plus and minus one maps, interchanges it with the opposite bracket.

The fact that this discussion illustrates is the existence of a forgetful functor from $\Disk_n^\fr$-algebras to Lie algebras at the level of $\infty$-categories. Indeed, according to \cite{Knudsen}, there is a commuting diagram of adjunctions $$\xymatrix{
\Alg_{\Disk_n^\fr}(\Ch_\mathbb{Q})\ar[dd]\ar[rr]&&\Alg_{\Lie}(\Ch_\mathbb{Q})\ar[dd]\ar@/_1.2pc/[ll]_{U_n}\\\\
\Ch_\mathbb{Q}\ar@/^1.2pc/[uu]^-{\mathbb{F}_n^\fr}\ar[rr]^-{[n-1]}&&\Ch_\mathbb{Q}\ar@/_1.2pc/[uu]_-{\Lie}\ar@/^1.2pc/[ll]^{[1-n]}.
}$$ Here $\Lie$ denotes the free Lie algebra functor.

The $\Disk_n^\fr$-algebra $U_n(\mathfrak{g})$ is known as the \emph{$n$-enveloping algebra} of $\mathfrak{g}$ (see \cite[4.6]{Gwilliam} for a discussion of the identification between $U_1$ and the usual universal enveloping algebra). The factorization homology of these algebras is computed in \cite{Knudsen}:

\begin{theorem}[K]\label{enveloping calculation 1}
There is a natural equivalence $$\xymatrix{\displaystyle\int_M U_n(\mathfrak{g})\ar[r]^-\sim&C^\Lie(\mathfrak{g}^{M^+}).
}$$
\end{theorem}

We pause briefly to explain the terms of the theorem. 

\begin{enumerate}
\item The $\infty$-category of differential graded Lie algebras has limits and is therefore cotensored over pointed spaces; we denote by $\mathfrak{g}^X$ the cotensor of the pointed space $X$ with the Lie algebra $\mathfrak{g}$. A model for this object is provided by \cite[4.8.3]{Hinich}:

\begin{proposition}\label{cotensor identification}
Let $X$ be a pointed finite CW complex. There is a natural equivalence $$\mathfrak{g}^X\simeq A_{PL}(X)\otimes\mathfrak{g}.$$
\end{proposition}

Here $A_{PL}$ denotes the functor of reduced $PL$ de Rham forms (see \cite[10(c)]{FHT}, for example), and the righthand side carries the canonical Lie bracket on the tensor product of a nonunital commutative algebra and a Lie algebra, which is defined by the formula $$[a\otimes v,b\otimes w]=(-1)^{|v||b|}ab\otimes[v,w].$$

\item The symbol $C^\Lie$ denotes the functor of Lie algebra chains. This coaugmented cocommutative coalgebra is defined abstractly via the monadic bar construction against the free Lie algebra monad, but it has a concrete incarnation as the \emph{Chevalley-Eilenberg complex} $$CE(\mathfrak{g})=(\Sym(\mathfrak{g}[1]), d_\mathfrak{g}+D),$$ where $D$ is defined as a coderivation by specifying that $$D(\sigma x\wedge\sigma y)=(-1)^{|x|}\sigma[x,y].$$ See \cite[6]{FresseKD} for a discussion of the comparison between the monadic bar construction and the Chevalley-Eilenberg complex. We remark that $CE(\mathfrak{g})$ is a coaugmented cocommutative differential graded coalgebra, and the resulting coproduct on $H^\Lie(\mathfrak{g})$ coincides with the one inherited from the monadic bar construction; indeed, both are induced by the diagonal $\mathfrak{g}\to \mathfrak{g}\oplus\mathfrak{g}$, which is a map of Lie algebras.
\end{enumerate}

The equivalence of Theorem \ref{enveloping calculation 1} specializes to a natural equivalence \[U_n(\mathfrak{g})\simeq C^\Lie(\mathfrak{g}^{(\mathbb{R}^n)^+})\] of $\Disk^\fr_n$-algebras. In this way, Theorem \ref{enveloping calculation 1} can be thought of as identifying an $n$-disk algebra refinement of the $\Disk_n^\fr$-algebra $U_n(\mathfrak{g})$, so that the expression $\int_MU_n(\mathfrak{g})$ is sensible for arbitrary, and not only framed, manifolds $M$.

Returning to the discussion that began this section, if $A$ is now an $n$-disk algebra rather than merely a $\Disk^\fr_n$-algebra, then $A$ determines a shifted Lie algebra in $O(n)$-modules, but now with $O(n)$ acting on the suspension coordinates. A full discussion of this phenomenon and the corresponding enveloping algebra is beyond the scope of this paper. Since the analogue of Theorem \ref{enveloping calculation 1} is true in that context, we will content ourselves with making it our definition.

As a matter of notation, if $X$ is a pointed $O(n)$-space and $\mathfrak{g}$ a Lie algebra in $O(n)$-modules, we denote the $O(n)$-invariants of $\mathfrak{g}^X$ by $\Map^{O(n)}(X,\mathfrak{g})$.

\begin{definition}\label{enveloping definition}
Let $\mathfrak{g}$ be a Lie algebra in $O(n)$-modules. The $n$-\emph{enveloping algebra} of $\mathfrak{g}$ is the $n$-disk algebra $$U_n(\mathfrak{g})=C^\Lie(\textstyle\Map^{O(n)}(\Fr_{(\mathbb{R}^n)^+},\mathfrak{g})).$$
\end{definition}

Here we take the frame bundle of the one-point compactification to be the cofiber $$\Fr_{M^+}=\xymatrix{\mathrm{cofib}(\Fr_{\bar M}|_{\partial \bar M}\ar[r]& \Fr_{\bar M})}$$ of $O(n)$-spaces, where $\bar M$ is a compact $n$-manifold with boundary whose interior is $M$ (see \cite[4.5.1]{AFZ} for a more invariant interpretation of this object). 

A choice of framing of $\mathbb{R}^n$ trivializes $\Fr_{(\mathbb{R}^n)^+}$, inducing an equivalence $$\textstyle\Map^{O(n)}(\Fr_{(\mathbb{R}^n)^+},\mathfrak{g})\simeq \mathfrak{g}^{(\mathbb{R}^n)^+},$$ which is even equivariant for the diagonal action of $O(n)$ on the target, so this definition specializes via Theorem \ref{enveloping calculation 1} to our earlier one when $\mathfrak{g}$ is an ordinary Lie algebra. 

The corresponding factorization homology calculation is the following:

\begin{proposition}\label{enveloping calculation 3}
There is a natural equivalence $$\xymatrix{\displaystyle\int_MU_n(\mathfrak{g})\ar[r]^-\sim&\textstyle C^\Lie(\Map^{O(n)}(\Fr_{M^+},\mathfrak{g}))}$$ for $M$ an $n$-manifold and $\mathfrak{g}$ a Lie algebra in $O(n)$-modules.
\end{proposition}
\begin{proof}
Since $C^\Lie$, as a left adjoint, preserves colimits, it suffices to exhibit an equivalence of Lie algebras $$\xymatrix{\displaystyle\int_M\mathfrak{g}^{(\mathbb{R}^n)^+}\ar[r]^-\sim& \Map^{O(n)}(\Fr_{M^+},\mathfrak{g}),}$$ which is supplied by \cite[4.5.3,\,4.7.1]{Francis}, since sifted colimits of Lie algebras are computed in $\Ch_\mathbb{Q}$ by \cite[2.1.16]{LurieX}.
\end{proof}

We close this section with a definition of a grading that will play an important role in what follows. Let $\mathfrak{g}$ be a differential graded Lie algebra with a \emph{weight decomposition} as a direct sum of complexes $\mathfrak{g}=\bigoplus_k\mathfrak{g}(k)$ with the property that $[v,w]\in \mathfrak{g}(r+s)$ when $v\in\mathfrak{g}(r)$ and $w\in\mathfrak{g}(s)$.

\begin{example}
A free Lie algebra has a canonical weight decomposition $$\Lie(V)=\bigoplus_{k\geq0}\Lie(k)\otimes_{\Sigma_k}V^{\otimes k}.$$
\end{example}

\begin{example}
If $\mathfrak{g}$ has a weight decomposition, then $A_{PL}(X)\otimes\mathfrak{g}$ carries a canonical weight decomposition for any space $X$.
\end{example}

\noindent Such a decomposition induces a \emph{weight grading} on the underlying graded vector space of $\Sym(\mathfrak{g}[1])$ of the Chevalley-Eilenberg complex. In fact, since we have assumed that the bracket and differential of $\mathfrak{g}$ each respect the weight decomposition, the Chevalley-Eilenberg differential applied to a monomial of pure weight $k$ again has pure weight $k$, so that $CE(\mathfrak{g})$ is a bicomplex. In this way, a weight decomposition of $\mathfrak{g}$ induces a weight grading on $H(U_n(\mathfrak{g}))$.

\section{Configuration Spaces}

\subsection{The main result} \label{proof one} In this section, we prove Theorem \ref{with grading} assuming the validity of several results, discussion of which is postponed for the sake of continuity, as the proofs involve different techniques from those used thus far. 

As a preliminary step, we have the following basic pair of observations:

\begin{proposition}
\begin{enumerate}
\item Let $K$ be an $O(n)$-module and $\underline{K}$ its underlying chain complex. There is a natural equivalence of framed $n$-disk algebras $$\mathbb{F}^\fr_n(\underline{K})\simeq \mathbb{F}_n(K).$$
\item Let $\mathfrak{g}$ be a Lie algebra in $O(n)$-modules and $\underline{\mathfrak{g}}$ its underlying Lie algebra. There is a natural equivalence of framed $n$-disk algebras $$U_n(\underline{\mathfrak{g}})\simeq U_n(\mathfrak{g}).$$
\end{enumerate}
\end{proposition}
\begin{proof} A choice of framing for $\mathbb{R}^n$ induces an $O(n)$-equivariant homotopy equivalence $$\xymatrix{\Emb^{\mathrm{fr}}(\amalg_k\mathbb{R}^n,\mathbb{R}^n)\times O(n)^k\ar[r]^-\sim& \Emb(\amalg_k\mathbb{R}^n,\mathbb{R}^n)},$$ whence from Proposition \ref{free identification} we have \begin{align*}\mathbb{F}_n(K)&\cong\bigoplus_{k\geq0}\bigg(\Emb^{\mathrm{fr}}(\amalg_k\mathbb{R}^n,\mathbb{R}^n)\times O(n)^k\bigg)\bigotimes_{\Sigma^k\ltimes O(n)^k}K^{\otimes k}\\
&\cong\bigoplus_{k\geq0}\bigg(\bigg(\Emb^{\mathrm{fr}}(\amalg_k\mathbb{R}^n,\mathbb{R}^n)\times O(n)^k\bigg)\bigotimes_{O(n)^k}K^{\otimes k}\bigg)_{\Sigma_k}\\
&\cong\bigoplus_{k\geq0}\Emb^{\mathrm{fr}}(\amalg_k\mathbb{R}^n,\mathbb{R}^n)\bigotimes_{\Sigma_k}K^{\otimes k}\\
&\cong\mathbb{F}_n^{\mathrm{fr}}(\underline{K}).
\end{align*} This proves (i), and (ii) is immediate from Definition \ref{enveloping definition}.
\end{proof}

\begin{remark}
Thinking topologically, the generic example of an $n$-disk algebra in spaces is an $n$-fold loop space on an $O(n)$-space $X$ (see \cite{SalvatoreWahl} or \cite{Wahl}). In this context, the statement is that, \emph{as an $n$-fold loop space}, the homotopy type of $\Omega^n X$ does not depend on the action of $O(n)$ on $X$. 
\end{remark}

Connecting (i) and (ii) is the following formal observation, which amounts to the statement that left adoints compose.

\begin{proposition}
Let $V$ be a chain complex. There is a natural equivalence $$\xymatrix{\mathbb{F}_n^{\mathrm{fr}}(V)\ar[r]^-\sim& U_n(\Lie(V[n-1]))}$$ of framed $n$-disk algebras, where $\Lie$ is the free Lie algebra functor.
\end{proposition}

This observation is a generalization of the familiar fact that the universal enveloping algebra of the free Lie algebra on a set of generators $S$ is free on $S$ as an associative algebra; however, equipped with the involution given by its Hopf algebra antipode, the universal enveloping algebra of the free Lie algebra on $S$ is \emph{not} the free algebra-with-involution on $S$. This classical fact illustrates the $n=1$ case of the general phenomenon that the free $n$-disk algebra on the trivial $O(n)$-module $V$ is \emph{not} the $n$-enveloping algebra of the free Lie algebra on $V$. As the following proposition shows, the $O(n)$-action must be twisted to restore the equivalence.

\begin{proposition}\label{comparison}
Let $K$ be an $O(n)$-module. There is a natural equivalence $$\xymatrix{\mathbb{F}_n(K)\ar[r]^-\sim& U_n(\Lie((\mathbb{R}^n)^+\otimes K[-1]))}$$ of $n$-disk algebras, where $(\mathbb{R}^n)^+\otimes K$ carries the diagonal $O(n)$-action.
\end{proposition}
\begin{proof}
First, we note that the unit $$\xymatrix{K\ar[r]&((\mathbb{R}^n)^+\otimes K)^{(\mathbb{R}^n)^+}}$$ of the tensor/cotensor adjunction is an equivalence of $O(n)$-modules. Indeed, it suffices to verify this in the case $K=\mathbb{Q}$, in which case the map induces the isomorphism $\mathbb{Q}\cong(\mathbb{Q}^{\det})^{\otimes 2}$ in homology. 

Now, composing this unit map with the natural inclusions $$\xymatrix{((\mathbb{R}^n)^+\otimes K)^{(\mathbb{R}^n)^+}\ar[r]&\Lie((\mathbb{R}^n)^+\otimes K[-1])^{(\mathbb{R}^n)^+}[1]\ar[r]& C^\Lie(\Lie((\mathbb{R}^n)^+\otimes K[-1])^{(\mathbb{R}^n)^+})}$$ of $O(n)$-modules, we obtain a map of $n$-disk algebras $$\xymatrix{\mathbb{F}_n(K)\ar[r]&U_n(\Lie((\mathbb{R}^n)^+\otimes K[-1]))}$$ from the universal property of the free algebra. It will suffice to show that this map is an equivalence upon passing to underlying $\Disk_n^\mathrm{fr}$-algebras, which follows from the previous two propositions and the (non-equivariant) equivalence $\mathbb{Q}[n]\simeq \widetilde{C}_*((\mathbb{R}^n)^+;\mathbb{Q})$.
\end{proof}

\begin{proof}[Proof of Theorem \ref{with grading}]
An equivalence of $n$-disk algebras induces an equivalence on passing to factorization homology. Using the indicated results, we obtain equivalences \begin{align*}
\bigoplus_{k\geq0}C_*(B_k(M);\mathbb{Q})&\simeq \int_M\mathbb{F}_n(\mathbb{Q}) & (\ref{homology})\\
&\simeq \int_MU_n(\Lie(\widetilde C_*((\mathbb{R}^n)^+)[-1]))& (\ref{comparison})\\
&\simeq\textstyle C^\Lie(\Map^{O(n)}(\Fr_{M^+},\Lie(\widetilde C_*((\mathbb{R}^n)^+)[-1])))&(\ref{enveloping calculation 3})\\
&\simeq\textstyle C^\Lie(\Map^{O(n)}(\Fr_{M^+},\Lie(\mathbb{Q}^{\det}[n-1])))&(\ref{equivariant formality})\\
&\simeq\textstyle C^\Lie(\Map^{C_2}(\widetilde{M^+},\Lie(\mathbb{Q}^{\sgn}[n-1])))\\
&\simeq\textstyle C^\Lie(H_c^{-*}(M,\Lie(\mathbb{Q}^{w}[n-1])))&(\ref{Lie formality}).
\end{align*}

Applying Proposition \ref{spectral sequence} to this equivalence of algebras, we obtain an isomorphism of spectral sequences. The weight and cardinality gradings of the two algebras pass to factorization homology, so that these spectral sequences are each trigraded. According to Proposition \ref{bigraded isomorphism}, the isomorphism preserves the extra grading on $E^2$ and hence on $E^\infty$.
\end{proof}

\subsection{Variations} In this section, we discuss the corresponding results for twisted homology and manifolds with boundary.

\begin{theorem}\label{twisted version} Let $M$ be an $n$-manifold.
\begin{enumerate}
\item If $n$ is even, there is an isomorphism of bigraded vector spaces $$\bigoplus_{k\geq0}H_*(B_k(M);\mathbb{Q}^w)\cong H^\Lie(H_{c}^{-*}(M;\Lie(\mathbb{Q}[n-1]))).$$
\item If $n$ is odd, there is an isomorphism of bigraded vector spaces $$\bigoplus_{k\geq0}H_*(B_k(M);\mathbb{Q}^w)[k]\cong H^\Lie(H_{c}^{-*}(M;\Lie(\mathbb{Q}[n]))).$$
\end{enumerate}
\end{theorem}
\begin{proof}
We imitate the proof of Theorem \ref{with grading}. In the even case, we have \begin{align*}
\bigoplus_{k\geq0}C_*(B_k(M);\mathbb{Q}^w)&\simeq \int_M\mathbb{F}_n(\mathbb{Q}^{\det}) & (\ref{twisted homology})\\
&\simeq \int_MU_n(\Lie(\widetilde C_*((\mathbb{R}^n)^+)\otimes\mathbb{Q}^{\det}[-1]))& (\ref{comparison})\\
&\simeq\textstyle C^\Lie(\Map^{O(n)}(\Fr_{M^+},\Lie(\widetilde C_*((\mathbb{R}^n)^+)\otimes\mathbb{Q}^{\det}[-1])))&(\ref{enveloping calculation 3})\\
&\simeq\textstyle C^\Lie(\Map^{O(n)}(\Fr_{M^+},\Lie((\mathbb{Q}^{\det})^{\otimes 2}[n-1])))&(\ref{equivariant formality})\\
&\simeq\textstyle C^\Lie(\Map^{O(n)}(\Fr_{M^+},\Lie(\mathbb{Q}[n-1])))\\
&\simeq\textstyle C^\Lie(\Lie(\mathbb{Q}[n-1])^{M^+})\\
&\simeq\textstyle C^\Lie(H_c^{-*}(M,\Lie(\mathbb{Q}[n-1])))&(\ref{Lie formality}),
\end{align*}
and the odd case is essentially identical. The same argument as in the proof of Theorem \ref{with grading} shows that the resulting isomorphism is bigraded.
\end{proof}

Now, if $M$ is a manifold with boundary, then $B_k(M)\simeq B_k(\mathring{M})$, since configuration spaces are isotopy functors. A more interesting configuration space in this context is the \emph{relative configuration space} $$B_k(M,\partial M):=\frac{B_k(M)}{\{(x_1,\ldots, x_k)\mid \exists i: x_i\in\partial M\}}.$$

From the point of view of factorization homology, the natural setting in which to study these spaces is that of the \emph{zero-pointed manifolds} of \cite{AFZ}, a class of pointed spaces that are manifolds away from the basepoint. Indeed, if $M$ is a manifold with boundary, then $M/\partial M$ is naturally a zero-pointed manifold.

The algebraic counterpart of a basepoint is an \emph{augmentation}.

\begin{definition}
An \emph{augmented $n$-disk algebra} is an $n$-disk algebra $A$ together with a map of $n$-disk algebras $\epsilon:A\to \mathbb{Q}$.
\end{definition}

\begin{example}
The free $n$-disk algebra $\mathbb{F}_n(K)$ is naturally augmented via the unique map of $O(n)$-modules $K\to 0$.
\end{example}

\begin{example}
The $n$-enveloping algebra $U_n(\mathfrak{g})$ is naturally augmented via the unique map of Lie algebras $\mathfrak{g}\to 0$.
\end{example}

The theory of factorization homology for zero-pointed $n$-manifolds with coefficients in augmented $n$-disk algebras is expounded at length in \cite{AFZ} and \cite{AF}. For us, what is important is that, if $M$ is a manifold with boundary, then the factorization homology of $M/\partial M$ is defined for any choice of augmented $n$-disk algebra; moreover, if $\partial M=\varnothing$, then $M/\partial M=M_+$, and the factorization homology of the zero-pointed manifold $M_+$ with coefficients in $\epsilon:A\to \mathbb{Q}$ is equivalent to the factorization homology of $M$ with coefficients in $A$ defined previously. 

Our arguments go through in this more general context.

\begin{theorem}\label{boundary case}
Let $M$ be an $n$-manifold with boundary. There is an isomorphism of bigraded vector spaces $$\bigoplus_{k\geq0}\widetilde{H}_*(B_k(M,\partial M);\mathbb{Q})\cong H^\Lie(H_{c}^{-*}(M;\Lie(\mathbb{Q}^w[n-1]))).$$
\end{theorem}
\begin{proof}
We explain the adjustments necessary in the proof of Theorem \ref{with grading}. First, \cite[2.4.1]{AF} guarantees the equivalence $$\bigoplus_{k\geq0}\widetilde{C}_*(B_k(M,\partial M);\mathbb{Q})\simeq \int_{M/\partial M}\mathbb{F}_n(\mathbb{Q}).$$ Second, it is immediate from its definition that the map of Proposition \ref{comparison} is a map of augmented $n$-disk algebras, so that $$\int_{M/\partial M}\mathbb{F}_n(\mathbb{Q}) \simeq \int_{M/\partial M}U_n(\Lie(\widetilde C_*((\mathbb{R}^n)^+)[-1])).$$ The proof of Proposition \ref{enveloping calculation 3} translates verbatim into the zero-pointed context, so that we have the further equivalence $$\int_{M/\partial M}U_n(\Lie(\widetilde C_*((\mathbb{R}^n)^+)[-1])) \simeq\textstyle C^\Lie(\Map^{O(n)}(\Fr_{M^+},\Lie(\widetilde C_*((\mathbb{R}^n)^+)[-1]))).$$ The remainder of the proof goes through unchanged.
\end{proof}

\begin{remark}
When $M$ has boundary, there are two obvious candidates for the orientation sheaf of $M$, namely the ordinary and the exceptional pushforwards of the orientation sheaf of the interior of $M$. We intend the former here.
\end{remark}

\subsection{Formulas}\label{formulas} In this section, we use Theorem \ref{with grading} and the Chevalley-Eilenberg complex to reproduce and extend the classical results on the rational homology of configuration spaces alluded to in the introduction. 

We remind the reader that the free Lie algebra on $\mathbb{Q}^w[r]$ is given as a graded vector space by $$\Lie(\mathbb{Q}^w[r])\cong\begin{cases}
\mathbb{Q}^w[r]\oplus\mathbb{Q}[2r]\quad& r\text{ odd}\\
\mathbb{Q}^w[r]\quad &r\text{ even.}
\end{cases}$$ When $r$ is odd, the only nonvanishing bracket is the isomorphism $(\mathbb{Q}^w[r])^{\otimes 2}\cong\mathbb{Q}[2r]$.

\begin{corollary}
If $n$ is odd, there is an isomorphism $$H_*(B_k(M);\mathbb{Q})\cong \Sym^k(H_*(M;\mathbb{Q})).$$
\end{corollary}
\begin{proof}
Since $n$ is odd, the Lie algebra in question is abelian, so that the Chevalley-Eilenberg complex has no differential, and the weight grading coincides with the usual grading of the symmetric algebra. The claim follows after replacing shifted, twisted, compactly supported cohomology with homology using Poincar\'{e} duality.
\end{proof}

This result is Theorem C of \cite{BCT} as formulated in dual form in Theorem 4 of \cite{FelixTanre}, in which the isomorphism on cohomology is shown to be an isomorphism of algebras.

\begin{corollary}\label{even corollary}
If $n$ is even, $H_*(B_k(M);\mathbb{Q})$ is isomorphic to the homology of the complex $$\bigg(\bigoplus_{i=0}^{\lfloor\frac{k}{2}\rfloor}\Sym^{k-2i}(H^{-*}_c(M;\mathbb{Q}^w)[n])\otimes\Sym^{i}(H^{-*}_c(M;\mathbb{Q})[2n-1]),\,D\bigg),$$ where the differential $D$ is defined as a coderivation by the equation $$D(\sigma^n\alpha\wedge\sigma^n\beta)=(-1)^{(n-1)|\beta|}\sigma^{2n-1}(\alpha\smile\beta).$$
\end{corollary}
\begin{proof}
It suffices by Theorem \ref{with grading} to identify the complex in question with the weight $k$ part of the Chevalley-Eilenberg complex for $\mathfrak{g}=H_c^{-*}(M;\Lie(\mathbb{Q}^w[n-1]))$, which as a graded vector space is given by $$\Sym(\mathfrak{g}[1])\cong\Sym(H_c^{-*}(M;\mathbb{Q}^w)[n])\otimes \Sym(H_c^{-*}(M;\mathbb{Q})[2n-1]),$$ with differential determined as a coderivation by the bracket of $\mathfrak{g}$, which is none other than the shifted cup product shown above, with the sign determined by the usual Koszul rule of signs. Since the cogenerators of the first tensor factor have weight 1 and those of the second tensor factor weight 2, the subcomplex of total weight $k$ is exactly the sum shown above. 
\end{proof}

When $M$ is closed, orientable, and nilpotent, we recover the linear and Poincar\'{e} dual of Theorem A of \cite{FT}, as formulated in Theorem 1 of \cite{FelixTanre}. When $M$ is a once-punctured surface, we recover Theorem C of \cite{BC}.

\begin{remark} The proofs of Theorem 1 and the even-dimensional half of Theorem 3 of \cite{FelixTanre} rely crucially on the results of \cite{FT} and thereby on the hypotheses of compactness, orientability, and nilpotence. At the time of writing, these hypotheses do not appear in the statements of the theorems. 

It follows from our results, however, that these theorems are true at the stated level of generality. Indeed, by Theorem 6 of \cite{FelixTanre}, the $\Sigma_k$-invariants of the $E_1$ page of the Cohen-Taylor-Totaro spectral sequence (see \cite{CohenTaylor} and \cite{Totaro}) coincide with the linear dual of the complex exhibited in Corollary \ref{even corollary}.

An analogous spectral sequence in the nonorientable case, possibly with twisted coefficients, is available due to \cite{Getzler2step} and \cite{GetzlerHodge} (see also \cite{Petersen}).
\end{remark}

We leave it to the reader to formulate the analogous results on twisted homology and those concerning the homology of the relative configuration spaces $B_k(M,\partial M)$, which follow in the same way from Theorems \ref{twisted version} and \ref{boundary case}, respectively. To the author's knowledge, the computation in the twisted case is new in all cases except when $M$ is orientable and $n$ is even, so that $B_k(M)$ is orientable, and the computation in the relative case is new in all cases except when $\partial M=\varnothing$.

\section{Coalgebraic Structure}\label{comultiplicative structure}

\subsection{Primitives and weight} Our present goal is to supply the first of the missing ingredients in the proof of the main theorem, namely the identification of the cardinality and weight gradings at the level of homology (see Definition \ref{cardinality grading} and the end of \S\ref{Lie algebras} for definitions of these gradings). We make this identification locally on $M$ in this section and globalize in the following section using a spectral sequence argument.

Let $K$ be an $O(n)$-module. We define the following maps: \begin{enumerate}
\item $\iota:K\to \mathbb{F}_n(K)$ is the map of $O(n)$-modules given by the unit of the free/forgetful adjunction;
\item $\eta:\mathbb{Q}\to \mathbb{F}_n(K)$ is the unit of $\mathbb{F}_n(K)$ as an $n$-disk algebra; 
\item $\delta:\mathbb{F}_n(K)\to \mathbb{F}_n(K)\otimes\mathbb{F}_n(K)$ is the map of $n$-disk algebras induced by the composite $$\xymatrix{K\ar[rr]^-\Delta&&K\oplus K\ar[rr]^-{\eta\otimes\iota+\iota\otimes\eta}&&\mathbb{F}_n(\mathbb{Q})\otimes\mathbb{F}_n(\mathbb{Q})},$$ where $\Delta$ is the diagonal and we have tacitly employed the canonical identifications $K\otimes\mathbb{Q}\cong K\cong\mathbb{Q}\otimes K$;
\item $\delta_M$ is the map on factorization homology induced by $\delta$ (resp. $\eta_M$, $\eta$).
\end{enumerate}

Note that we have suppressed the choice of $K$ from the notation. 

Although we will only use the case $M=\mathbb{R}^n$ here, we record the following result for its inherent interest:

\begin{proposition}
The maps $H(\delta_M)$ and $H(\eta_M)$ endow $H\bigg(\displaystyle\int_M\mathbb{F}_n(K)\bigg)$ with the structure of a coaugmented cocommutative coalgebra.
\end{proposition}
\begin{proof}
The functor $\int_M$ is symmetric monoidal in the algebra variable by \cite[5.5.3.2]{HA}, so it suffices to verify the claim in the case $M=\mathbb{R}^n$. The required axioms all follow from the universal property of the free algebra; we spell out the argument for coassociativity, leaving the remainder to the reader. 

Consider the cubical diagram $$\xymatrix{&K\ar[rr]^<<<<<<<<<<<<<<<<<<<<<<<<<<\Delta\ar[dd]_>>>>>>>>>{\Delta}|!{[d]}\hole\ar[dl]_-\eta\ar[rr]&&K\oplus K\ar[dd]^{\Delta\oplus1}\ar[dl]_-{\eta\otimes\iota+\iota\otimes\eta\quad}\\
\mathbb{F}_n(K)\ar[rr]^<<<<<<<<<<<<<<<<<<<<<<<<<<<<<<<\delta\ar[dd]_-{\delta}&& \mathbb{F}_n(K)\otimes\mathbb{F}_n(K)\ar[dd]^>>>>>>>>>>>>>>>{\delta\otimes1}\\
&K\oplus K\ar[dl]_{\eta\otimes\iota+\iota\otimes\eta\,\,\,\,}\ar[rr]^<<<<<<<<<<<<<<<<<<<<<<{1\oplus\Delta}|!{[r]}\hole&&K\oplus K\oplus K\ar[dl]^{\qquad\qquad\eta\otimes\iota\otimes\iota+\iota\otimes\eta\otimes\iota+\iota\otimes\iota\otimes\eta}
\\
\mathbb{F}_n(K)\otimes\mathbb{F}_n(K)\ar[rr]^-{1\otimes\delta}&&\mathbb{F}_n(K)\otimes\mathbb{F}_n(K)\otimes\mathbb{F}_n(K).
}$$ It will suffice to show that the square diagram given by the front face of the cube commutes in the $\infty$-category of $n$-disk algebras, since this square witnesses coassociativity after applying factorization homology and passing to the homotopy category of chain complexes. Applying the universal property of the free algebra, the required commutativity is equivalent to commutativity as a diagram of $O(n)$-modules after precomposing with $\eta$. 
By a standard diagram chase, it suffices to verify that the remaining five faces each commute. \begin{itemize}
\item The left and top face commute by the definition of $\delta$.
\item The back face commutes by the universal property of the direct sum, considered as the $\infty$-categorical product.
\item The right and bottom face commute by the definition of $\delta$ and the universal property of the direct sum, considered as the $\infty$-categorical coproduct.
\end{itemize}
\end{proof}

Although we have defined this coalgebra structure in abstract terms, it has an appealing geometric interpretation, which we discuss in \S\ref{splitting points} below. 

When $M=\mathbb{R}^n$, the same homology is also an algebra, and even commutative for $n\geq2$. Since $\delta$ is a map of $n$-disk algebras, $H(\mathbb{F}_n(\mathbb{Q}))$ inherits the structure of a \emph{bialgebra}, and in fact a Hopf algebra, although we will not make use of the antipode.

For the duration of this section, we make the abbreviation $\mathfrak{g}(K):=\Lie((\mathbb{R}^n)^+\otimes K[-1])^{(\mathbb{R}^n)^+}$.

\begin{proposition}\label{coalgebra iso}
The isomorphism on homology induced by the equivalence of Proposition \ref{comparison} is an isomorphism of bialgebras.
\end{proposition}
\begin{proof}
Denote by $\varphi$ the equivalence $$\xymatrix{\mathbb{F}_n(K)\ar[r]^-\sim&C^\Lie(\mathfrak{g}(K))}$$ of Proposition \ref{comparison}. Since $\varphi$ is a map of $n$-disk algebras, the induced map on homology is a map of algebras; therefore, it will suffice to show that this map is also a map of coalgebras.

Consider the cubical diagram $$\xymatrix{&K\ar[rr]\ar[dd]_>>>>>>>>>{\Delta}|!{[d]}\hole\ar[dl]_-\eta\ar[rr]&&\mathfrak{g}(K)[1]\ar[dd]^{\Delta}\ar[dl]\\
\mathbb{F}_n(K)\ar[rr]^<<<<<<<<<<<<<<<<<<<<<<<<<<<<<<<\varphi\ar[dd]_-{\delta}&& C^\Lie(\mathfrak{g}(K))\ar[dd]^>>>>>>>>>>>>>>>{\gamma}\\
&K\oplus K\ar[dl]_{\eta\otimes\iota+\iota\otimes\eta\,\,\,\,}\ar[rr]|!{[r]}\hole&&\mathfrak{g}(K)[1]\oplus\mathfrak{g}(K)[1]\ar[dl]
\\
\mathbb{F}_n(K)\otimes\mathbb{F}_n(K)\ar[rr]^-{\varphi\otimes\varphi}&&C^\Lie(\mathfrak{g}(K))\otimes C^\Lie(\mathfrak{g}(K)),
}$$ where $\gamma$ denotes the comultiplication on Lie algebra chains. As before, we wish to show that the front face commutes in the $\infty$-category of $n$-disk algebras, and, as before, this reduces to checking the commutativity of the remaining five faces in the $\infty$-category of $O(n)$-modules. \begin{itemize}
\item The left face commutes by the definition of $\delta$.
\item The back face commutes by functoriality of the diagonal.
\item The top face commutes by the definition of $\varphi$.
\item The bottom face commutes by the definition of $\varphi$ and the universal property of the direct sum, considered as the categorical coproduct.
\item The right face commutes because the functor $C^\Lie$ is Cartesian monoidal.
\end{itemize}
\end{proof}

This bialgebra is a familiar one, and the various components of its structure interact predictably with the bigradings.

\begin{proposition}
\begin{enumerate}
\item There are isomorphisms $$H^\Lie(\mathfrak{g}(K))\cong \Sym(H(\mathfrak{g}(K))[1])\cong H(\mathbb{F}_n(K))$$ of graded bialgebras, where $\Sym$ is equipped with the standard product and coproduct.
\item The product of $H(\mathbb{F}_n(K))$ preserves the cardinality grading.
\item The coproduct of $H(\mathbb{F}_n(K))$ preserves the cardinality grading.
\item The product of $H^\Lie(\mathfrak{g}(K))$ preserves the weight grading.
\item The coproduct of $H^\Lie(\mathfrak{g}(K))$ preserves the weight grading.
\end{enumerate}
\end{proposition}
\begin{proof}
\begin{enumerate}
\item We note that $\mathfrak{g}(K)$ is a formal Lie algebra, since the pointed space $(\mathbb{R}^n)^+$ is formal; moreover, since $H(\mathfrak{g}(K))$ is abelian, there is no differential in the Chevalley-Eilenberg complex, so we have isomorphisms of coaugmented coalgebras $$H^\Lie(\mathfrak{g}(K))\cong H^\Lie(H(\mathfrak{g}(K)))\cong\Sym(H(\mathfrak{g}(K))[1]).$$ From the discussion of \S\ref{Lie algebras}, the product of $H^\Lie(\mathfrak{g}(K))$ is the map induced on Lie algebra homology by the $n$-disk algebra structure map of $\mathfrak{g}(K)$ corresponding to any embedding $\mathbb{R}^n\amalg\mathbb{R}^n\to \mathbb{R}^n$, and any such structure map induces the fold map $$\xymatrix{H(\mathfrak{g}(K))\oplus H(\mathfrak{g}(K))\ar[r]^-+&H(\mathfrak{g}(K))}$$ at the level of homology. Likewise, the coproduct is induced by the diagonal, and we recognize the standard bialgebra structure on $\Sym$. The second isomorphism now follows by Proposition \ref{coalgebra iso}.
\item The cardinality grading is natural, and the product is the map induced on homology by the $n$-disk algebra structure map corresponding to any embedding $\mathbb{R}^n\amalg\mathbb{R}^n\to \mathbb{R}^n$.
\item Since the coproduct preserves the cardinality grading on generators by definition, the claim follows from (ii) and the fact that $H(\mathbb{F}_n(K)$ is a bialgebra.
\item The claim is immediate from the definition of the weight grading.
\item Since the coproduct preserves the cardinality grading on generators by definition, the claim follows from (iv) and the fact that $H^\Lie(\mathfrak{g}(K))$ is a bialgebra.
\end{enumerate}
\end{proof}

The desired identification of bigradings now follows easily.

\begin{proposition}\label{bigraded isomorphism}
In the case $M=\mathbb{R}^n$, the isomorphisms of Theorems \ref{with grading}, \ref{twisted version}, and \ref{boundary case} are isomorphisms of bigraded vector spaces.
\end{proposition}
\begin{proof}
We present the argument for Theorem \ref{with grading}, the others being essentially identical.

We follow the convention that a subscript indicates homological degree, a generator decorated with a tilde has weight 2, and an unadorned generator has weight 1. There is an isomorphism of bialgebras $H(\mathbb{F}_n(\mathbb{Q}))\cong\Sym(V_n)$, where $$V_n= \begin{cases}\mathbb{Q}\langle x_0\rangle&\quad n\text{ odd}\\
\mathbb{Q}\langle x_0, \tilde{y}_{n-1}\rangle&\quad n\text{ even}.
\end{cases}$$ 

Identifying both sides of the isomorphism of Theorem \ref{with grading} with $\Sym(V_n)$, Proposition \ref{coalgebra iso} permits us to view this isomorphism as an automorphism $f$ of this graded bialgebra. Now, as a morphism of graded coalgebras, $f$ takes primitives to primitives, so that there is an induced map $f|_{V_n}:V_n\to V_n$ of graded vector spaces, which we claim is a bigraded isomorphism. In the case of odd $n$, the claim is implied by the injectivity of $f|_{V_n}$, while in the even case we note that, for degree reasons, $f(x)$ is a scalar multiple of $x$ and $f(\tilde y)$ is a scalar multiple of $\tilde y$. By injectivity, this scalar is nonzero, and we conclude that $f|_{V_n}$ is a bigraded isomorphism. 

Now, since $f$ is also a map of algebras, we have the equation $f(x_1\cdots x_r)=f(x_1)\cdots f(x_r)$, which, together with the previous paragraph, shows that $f$ preserves weight on monomials. Since monomials form a bihomogeneous basis and $f$ is linear, the proof is complete.
\end{proof}

\subsection{Interlude: splitting configurations}\label{splitting points}

Configuration spaces of different cardinalities are interrelated by splitting and forgetting maps inherited from the Cartesian product via the embedding $\Conf_k(M)\to M^k$. This rich structure invites an inductive way of thinking that appears in one form or another in essentially every classical approach to these spaces; see \cite{Arnold} and \cite{FN} for the origins of this approach and \cite{Church} for a modern implementation.

In the setting of factorization homology, the importance of these splitting maps is that they assemble to form a coproduct, a shadow of which we have seen in the previous section, endowing $\mathbb{F}_n(K)$ with the structure of an $n$-disk algebra \emph{in cocommutative coalgebras.} We will not need the full force of this statement, nor will we need the geometric interpretation of this coalgebra structure; nevertheless, we devote the remainder of this section to elucidating this interpretation, both for its general interest and for the motivation it provides for our proof of homological stability.

\begin{remark}
The constructions of this section are valid in more general stable settings than chain complexes, including the symmetric monoidal $\infty$-category of spectra with smash product. We intend to return to this setting in future work.
\end{remark}

The basic ingredient is the collection of natural transformations $$s_{i,j}:\Conf^\fr_k\to \Conf^\fr_i\times\Conf^\fr_j,$$ defined whenever $i+j=k$, which make the diagram $$\xymatrix{
\Conf^\fr_k(M)\ar[d]\ar[r]^-{(s_{i,j})_M}&\Conf_i^\fr(M)\times\Conf_j^\fr(M)\ar[d]\\
\prod_k\Fr_M\ar[r]^-\cong&\prod_i\Fr_M\times\prod_j\Fr_M
}$$ commute; in other words, $$(s_{i,j})_M(x_1,\ldots, x_k)=\big((x_1,\ldots, x_i),(x_{i+1},\ldots, x_k)\big).$$ Given an $O(n)$-module $K$, then, we have maps $$\xymatrix{s^K_{i,j}:\Conf_k^\fr\otimes K^{\otimes k}\ar[r]^-{\delta_{i,j}\otimes 1}&(\Conf_i^\fr\times\Conf_j^\fr)\otimes K^{\otimes k}\ar[r]^-\simeq&\Conf_i^\fr\otimes K^{\otimes i}\otimes\Conf_j^\fr\otimes K^{\otimes j}},$$ which are $(\Sigma_i\times\Sigma_j)\ltimes O(n)^k$-equivariant. Taking $O(n)^k$-coinvariants and using that induction is right adjoint to restriction for the inclusion $\Sigma_i\times\Sigma_j\to \Sigma_k$, we obtain by adjunction a $\Sigma_k$-equivariant map $$\xymatrix{\tilde s_{i,j}^K:\Conf_k^\fr\displaystyle\bigotimes_{O(n)^k} K^{\otimes k}\ar[r]&\mathrm{Ind}_{\Sigma_i\times\Sigma_j}^{\Sigma_k}\bigg(\Conf_i^\fr\displaystyle\bigotimes_{O(n)^i} K^{\otimes i}\otimes\Conf_j^\fr\bigotimes_{O(n)^j} K^{\otimes j}}\bigg).$$ Finally, taking $\Sigma_k$-coinvariants and summing over $k, i,$ and $j$, we obtain a map \[s^K: \displaystyle\bigoplus_{k\geq0}\bigg(\Conf_k^\fr\displaystyle\bigotimes_{\Sigma_k\ltimes O(n)^k} K^{\otimes k}\bigg)\to\displaystyle\bigoplus_{k}\bigoplus_{i+j=k}\bigg(\Conf_i^\fr\displaystyle\bigotimes_{\Sigma_i\ltimes O(n)^i} K^{\otimes i}\otimes\Conf_j^\fr\bigotimes_{\Sigma_j\ltimes O(n)^j} K^{\otimes j}\bigg).\] Collecting terms and restricting to $\Disk_n$, we recognize this as a monoidal natural transformation $$\xymatrix{s^K:\mathbb{F}_n(K)\ar[r]&\mathbb{F}_n(K)\otimes\mathbb{F}_n(K).}$$ 

The proof of homological stability given in the next section is completely internal to the Chevalley-Eilenberg complex, but the motivation behind it comes from thinking of the symmetric coproduct, given by splitting monomials in all possible ways, as corresponding to this geometric coproduct, given by splitting configurations in all possible ways. To see the connection, we recall that, in the approach of \cite{Church}, stability is induced by the transfer maps $$\xymatrix{
H_*(\Conf_{k+1}(M);\mathbb{Q})\ar[rr]^-{\sum_i(p_i)_*}\ar[d]&&H_*(\Conf_k(M);\mathbb{Q})\ar[d]\\
H_*(B_{k+1}(M);\mathbb{Q})\ar@{-->}[rr]^-{\mathrm{tr}}&&H_*(B_k(M);\mathbb{Q}),
}$$ where $p_i$ denotes the projection that forgets $x_i$. In terms of our splitting maps, we have a factorization $$\xymatrix{
\Conf_{k+1}(M)\ar[d]_-{\sigma_i}\ar[r]^-{p_i}&\Conf_k(M)\\
\Conf_{k+1}(M)\ar[r]^-{s_{1,k}}&M\times \Conf_k(M)\ar[u]
}$$ where $\sigma_i$ denotes the permutation that moves $x_i$ to the first position while maintaining the relative order of the remaining points, and the unmarked arrow is the projection. The composite $s_{1,k}\sigma_i$ is a component of the coproduct defined above, and the projection away from the $M$ factor corresponds at the level of homology to evaluating against the unit in $H^0(M;\mathbb{Q})$. Together, these observations suggest that homological stability should be induced taking a \emph{cap product}. We realize this idea in the next section.

\subsection{Stability}\label{stability section}

This section assembles the proof of Theorem \ref{stability theorem}. Throughout, unless otherwise noted, $M$ will be connected, without boundary, and of dimension $n>1$. For the sake of brevity, we make the abbreviation $$\mathfrak{g}_M=H_c^{-*}(M;\Lie(\mathbb{Q}^w[n-1])).$$

Let $\lambda\in H^0(M)$ denote the multiplicative unit. We view this cohomology class as a functional on $H_0(M)\cong H_c^{0}(M;\mathbb{Q}^w)[n]$ and hence, extending by zero, on $CE(\mathfrak{g}_M)$, since the former is canonically a summand of the underlying bigraded vector space of the latter. Thus we may contemplate the \emph{cap product} with this element, denoted $\lambda\frown(-)$, which is defined as the composite$$\xymatrix{CE(\mathfrak{g}_M)\cong\mathbb{Q}\otimes CE(\mathfrak{g}_M)\ar[r]^-{\lambda\otimes\gamma}&CE(\mathfrak{g}_M)^\vee\otimes CE(\mathfrak{g}_M)\otimes CE(\mathfrak{g}_M)\ar[rr]^-{\langle-,-\rangle\otimes\mathrm{id}}&&CE(\mathfrak{g}_M)}.$$

Denote by $p\in H_c^n(M;\mathbb{Q}^w)\subset CE(\mathfrak{g}_M)$ the Poincar\'{e} dual of a point in $M$, which is well-defined since $M$ is connected. Extend the set $\{1,p\}$ once and for all to a bihomogeneous basis $\mathcal{B}$ for $\mathfrak{g}_M[1]$. Then the set of nonzero monomials in elements of $\mathcal{B}$ form a bihomogenous basis for $CE(\mathfrak{g}_M)$, and, under the resulting identification of this vector space with its dual, $\lambda$ is identified with the dual functional to $p$. Since $p$ is closed of degree 0 and weight 1, we conclude the following:

\begin{proposition}
$\lambda\frown(-)$ is a chain map of degree 0 and weight $-1$.
\end{proposition}

There is a simple formula describing this map. Here and throughout, when we speak of divisibility, multiplication, and differentiation in the Chevalley-Eilenberg complex, we refer only to the formal manipulation of bigraded polynomials; in particular, $CE(\mathfrak{g}_M)$ is \emph{not} in general a differential graded algebra.

\begin{proposition}\label{cap formula} The formula
$$\lambda \frown x=\displaystyle\frac{dx}{dp}$$ holds for all $x\in CE(\mathfrak{g}_M)$.
\end{proposition}
\begin{proof}
Both sides are linear, so the claim is equivalent to the equality $$\lambda\frown p^ry=rp^{r-1}y$$ whenever $r\geq0$ and $y$ is a monomial in elements of $\mathcal{B}\setminus\{p\}$. There are now two cases.

The first case is when $y$ is a scalar, in which case we may assume by linearity that $y=1$, so that $x=p^r$, and \begin{align*}\gamma(x)&=\gamma(p)^r\\
&=(p\otimes 1+1\otimes p)^r\\
&=\sum_{i=0}^r\binom{r}{i}p^i\otimes p^{r-i},
\end{align*} so that $$\lambda\frown x=\sum_{i=0}^r\binom{r}{i}\langle \lambda,p^i\rangle p^{r-i}=rp^{r-1}.$$

The second is when $y$ is a monomial in $\mathcal{B}\setminus\{1,p\}$, in which case we may write $$\gamma(y)=y\otimes 1+1\otimes y+\sum_j y_j\otimes y_j'$$ with $y_j$ and $y_j'$ monomials in $\mathcal{B}\setminus\{1,p\}$. Then we have \begin{align*}\gamma(p^ry)&=\gamma(p)^r\gamma(y)\\
&=(p\otimes 1+1\otimes p)^r(y\otimes 1+1\otimes y+\sum_j y_j\otimes y_j')\\
&=\sum_{i=0}^r\binom{r}{i}\bigg(p^iy\otimes p^{r-i}+p^i\otimes p^{r-i}y+\sum_jp^iy_j\otimes p^{r-i}y_j'\bigg),
\end{align*} whence $$\lambda\frown p^ry=\sum_{i=0}^r\binom{r}{i}\bigg(\langle \lambda,p^iy\rangle p^{r-i}+\langle \lambda, p^{i}\rangle p^{r-i}y+\sum_j\langle \lambda, p^iy_j\rangle p^{r-i}y_j'\bigg)=rp^{r-1}y,$$ since $p^iy$ is not a scalar multiple of $p$ for any $i$, nor is $p^i y_j$ a scalar multiple of $p$ for any $(i,j)$.
\end{proof}

\begin{corollary}\label{surjective}
The chain map $\lambda\frown(-)$ is surjective.
\end{corollary}
\begin{proof}
It suffices to show that a general monomial in elements of $\mathcal{B}$ lies in the image. Such a monomial may be written as $p^ry$ with $r\geq0$ and $y$ a monomial in elements of $\mathcal{B}\setminus\{p\}$. We than have $$\frac{d}{dp}\bigg(\frac{1}{r+1}p^{r+1}y\bigg)=p^ry.$$
\end{proof}

The central observation behind our approach to stability is the following.

\begin{proposition}\label{divisibility}
Let $x$ be a nonzero monomial in $CE(\mathfrak{g}_M)$. Then $x$ is divisible by $p$ provided
\begin{itemize}
\item $\mathrm{wt}(x)>|x|+1$ and $M$ is an orientable surface,  or
\item $\mathrm{wt}(x)>|x|$ and $M$ is not an orientable surface.
\end{itemize}
\end{proposition}
\begin{proof}
Suppose $\mathrm{wt}(x)>|x|$, and write $x=x_1\cdots x_r$ with $x_i\in \mathcal{B}$. Then $\mathrm{wt}(x_j)>|x_j|$ for some $j$. Since $x_j\in \mathfrak{g}_M[1]$, the weight of this element is either 1 or 2. 

In the first case, $x_j\in H_c^{-*}(M;\mathbb{Q}^w)[n]$, and we have $$|x_j|<\mathrm{wt}(x_j)=1\implies |x_j|=0,$$ since $H_c^{-*}(M;\mathbb{Q}^w)[n]$ is concentrated in degrees $0\leq *\leq n$. But $H_c^n(M;\mathbb{Q}^w)$ is one-dimensional on the class $p$, so $x_j$ is a scalar multiple of $p$.

In the second case, $x_j\in H_c^{-*}(M;\mathbb{Q})[2n-1]$, and we have $|x_j|<2$. Since $H_c^{-*}(M;\mathbb{Q})[2n-1]$ is concentrated in degrees $n-1\leq *\leq 2n-1$, we conclude that $x_j=0$ provided $n\neq 2$ (recall that we have already assumed $n>1$). Thus $x=0$, which is a contradiction. This proves the claim when $M$ is not a surface.

If $M$ is a non-orientable surface, then $H_c^{2}(M;\mathbb{Q})\cong H_0(M;\mathbb{Q}^w)=0,$ so that $H_c^{-*}(M;\mathbb{Q})[2n-1]$ is concentrated in degrees 2 and 3. Thus, in this case as well, we have a contradiction.

Assume now that $M$ is an orientable surface and $\mathrm{wt}(x)>|x|+1$. As before, write $x=x_1\ldots x_r$ and choose $x_j$ with $\mathrm{wt}(x_j)>|x_j|$, and assume that $x_j$ is not a scalar multiple of $p$. Then by the argument above, $\mathrm{wt}(x_j)=2$, so $|x_j|=1$, since $H_c^{-*}(M;\mathbb{Q})[3]$ is concentrated in degrees $1\leq*\leq 3$. 

Now, the monomial $x'=x_1\cdots \hat x_j\cdots x_r$ has the property that $$\mathrm{wt}(x')=\mathrm{wt}(x)-2>|x|-1=|x'|,$$ so there is some $x_i$ with $i\neq j$ and $\mathrm{wt}(x_i)>|x_i|$. If $x_i$ is a scalar multiple of $p$, we are finished; otherwise, repeating the same argument shows that $x_i$ has degree 1. But $H_c^{2}(M;\mathbb{Q})\cong H_0(M;\mathbb{Q})$ is one-dimensional, so that $x_i$ is a scalar multiple of $x_j$, and $x$ is divible by $x_j^2$. Since $x_j$ is of odd degree this implies that $x=0$, a contradiction.
\end{proof}

We are now equipped to prove Theorem \ref{stability theorem}. Denote by $C(k)$ the subcomplex of the Chevalley-Eilenberg complex spanned by the weight $k$ monomials. Taking the cap product with $1$ restricts to a map $\Phi_k:C(k+1)\to C(k)$, and we aim to show that this map induces an isomorphism in homology in the specified range.

Recall that the $r$th \emph{brutal truncation} of a chain complex $V$ is the chain complex $\tau_{\leq r}V$ whose underlying graded vector space is $$(\tau_{\leq r}V)_i=\begin{cases}
V_i\qquad &i\leq r\\
0\qquad &\text{else.}
\end{cases}$$ and whose differential is the restriction of the differential of $V$. Truncation is a functor on chain complexes in the obvious way.

We make use of the following elementary fact:

\begin{proposition}
Let $f:V\to W$ be a surjective chain map such that $\tau_{\leq r}f$ is a chain isomorphism. Then $f$ is a homology isomorphism through degree $r$ and a homology surjection in degree $r+1$.
\end{proposition}
\begin{proof}
From the definition of the brutal truncation, it is immediate that $f$ is a homology isomorphism through degree $r-1$. Moreover, $f$ induces a bijection on $r$-cycles and an injection on $r$-boundaries. 

To show that $f$ is a homology isomorphism in degree $r$, it suffices to show that $f^{-1}(w)$ is a boundary if $w\in W_r$ is a boundary. Write $w=du$; then, by surjectivity, there is some $\tilde u\in V_r$ such that $f(\tilde u)=u$, and $$f(d\tilde u)=df(\tilde u)=du=w\implies f^{-1}(w)=d\tilde u,$$ as desired.

To show that $f$ is a homology surjection in degree $r+1$, let $v\in W_{r+1}$ be a cycle. By surjectivity, $v=f(\tilde v)$, and it will suffice to show that $\tilde v$ is a cycle, for which we have $$f(d\tilde v)=df(\tilde v)=dv=0\implies d\tilde v=f^{-1}(0)=0.$$
\end{proof}

\begin{proof}[Proof of Theorem \ref{stability theorem}]
Assume first that $M$ is not an orientable surface. By the previous proposition and Proposition \ref{surjective}, we are reduced to showing that $\tau_{\leq k}\Phi_k$ is a chain isomorphism. To see this, let $x\in \tau_{\leq k}C(k+1)$ be a monomial. Then $\mathrm{wt}(x)>|x|$, so that $x=p^ry$ with $r>0$ and $y$ a monomial in $\mathcal{B}\setminus\{p\}$ by Proposition \ref{divisibility}. By Proposition \ref{cap formula}, $\Phi_k(x)=rp^{r-1}y,$ so $\tau_{\leq k}\Phi_k$ maps distinct elements of our preferred basis for $C(k+1)$ to nonzero scalar multiples of distinct elements of our preferred basis for $C(k)$, which implies that $\tau_{\leq k}\Phi_k$ is injective. But $\Phi$ and hence $\Phi_k$ are surjective by Proposition \ref{surjective}, so $\tau_{\leq k}\Phi_k$ is as well.

Assume now that $M$ is an orientable surface. For the same reason, we are reduced to showing that $\tau_{k-1}\Phi_k$ is a chain isomorphism, which is accomplished by the same argument, using the other half of Proposition \ref{divisibility}.
\end{proof}

\begin{remark}
Let $\mathbb{K}$ denote the Klein bottle. As shown in \S\ref{examples}, $$\dim H_*(B_k(\mathbb{K});\mathbb{Q})=\begin{cases}
1\qquad &i\in\{0,1,2, k+1\}\\
2\qquad &3\leq i\leq k\\
0\qquad &\text{else.}
\end{cases}$$ In particular, $H_{k+1}(B_{k+1}(\mathbb{K});\mathbb{Q})\ncong H_{k+1}(B_{k}(\mathbb{K});\mathbb{Q})$, so that our bound is sharp in the sense that no better stable range holds for all manifolds that are not orientable surfaces.
\end{remark}

\begin{remark}
If $M$ is orientable and $H_*(M;\mathbb{Q})=0$ for $1\leq *\leq r-1$, then $H_c^{-*}(M;\mathbb{Q})=0$ for $n-r+1\leq -*\leq n-1$, and the argument of Lemma \ref{divisibility} shows that a monomial $x$ is divisible by $p$ provided its weight is greater than $\frac{|x|}{r}+1$. This improved estimate leads to an improved stable range, as in \cite[4.1]{Church}.
\end{remark}

\begin{remark}
In \cite{MillerKupers}, factorization homology is used to obtain homological stability results for various constructions on open manifolds. The approach there is through certain ``partial algebras'' and appears unrelated to ours.
\end{remark}

\section{Examples}\label{examples}

We now present a selection of computations illustrating the following general procedure for determining the rational homology of the configuration space of $k$ points in an $n$-manifold $M$:

\begin{enumerate}
\item compute the compactly supported cohomology of $M$, twisted if necessary;
\item compute the Lie algebra homology of $H_c^{-*}(M;\Lie(\mathbb{Q}^w[n-1]))$;
\item count basis elements of weight $k$.
\end{enumerate}

\noindent It is worth noting that the Chevalley-Eilenberg complex allows one to obtain answers simultaneously for all $k$, reducing an infinite sequence of computations to one. 

The computations of this section are all relatively elementary, and one can do better with more effort. In \cite{Gabe}, this approach is used to determine the Betti numbers of $B_k(\Sigma)$ for every surface $\Sigma$.

\begin{convention} In the following examples, a variable decorated with a tilde has weight 2, while an unadorned variable has weight 1.
\end{convention}

\subsection{Punctured Euclidean space}

As a warm-up and base case, we recover the classical computation of $H_*(B_k(\mathbb{R}^n);\mathbb{Q})$. Since there are no cup products in the compactly supported cohomology of $\mathbb{R}^n$, there are no differentials in the corresponding Chevalley-Eilenberg complex. Thus $H_*(B_k(\mathbb{R}^n);\mathbb{Q})$ is identified with the subspace of $\mathbb{Q}[x]$ spanned by $x^k$ when $n$ is odd, while for $n$ even the identification is with the subspace of $$\mathbb{Q}[x]\otimes\Lambda[\tilde x],\qquad |x|=0,\,|\tilde x|=n-1$$ spanned by elements of weight $k$, a basis for which is given by $\{x^k, x^{k-2}\tilde x\}$. We conclude, for all $k>1$, that $$H_*(B_k(\mathbb{R}^n);\mathbb{Q})\cong\begin{cases}
\mathbb{Q}\qquad\qquad& n\text{ odd}\\
\mathbb{Q}\oplus\mathbb{Q}[n-1] &n\text{ even}.
\end{cases}$$

Now, choose $\bar p=\{p_1,\ldots, p_m\}\in\mathbb{R}^n$. There is a homotopy equivalence $(\mathbb{R}^n\setminus\bar p)^+\simeq S^{n}\vee (S^1)^{\vee m},$ so that $H_c^{-*}(\mathbb{R}^n\setminus\bar p;\mathbb{Q})\cong\mathbb{Q}^m[-1]\oplus\mathbb{Q}[-n].$ There are no cup products, so there can be no differentials.

If $n$ is odd, Theorem \ref{with grading} identifies $H_*(\mathbb{R}^n\setminus\bar p;\mathbb{Q})$ with the weight $k$ part of $$\mathbb{Q}[x,y_1,\ldots, y_m],\qquad |x|=0,\, |y_i|=n-1,$$ and an easy induction now shows that $$ \dim H_*(B_k(\mathbb{R}^n\setminus\bar p);\mathbb{Q})=\begin{cases}
\binom{m+i-1}{i}\quad &*=i(n-1), \,\,0\leq i\leq k\\
0\quad &\text{else.}
\end{cases}$$ (it is helpful to recall that $\binom{m+i-1}{i}$ is the number of ways to choose $i$ not necessarily distinct elements from a set of $m$ elements).

If $n$ is even, then the corresponding vector space is the weight $k$ part of $$\mathbb{Q}[x, \tilde y_1,\ldots, \tilde y_m]\otimes\Lambda[\tilde x, y_1,\ldots, y_m],\qquad |x|=0,\,|y_i|=|\tilde x|=n-1,\, |\tilde y_i|=2n-2.$$ Counting inductively in terms of less punctured Euclidean spaces, one finds that $$H_*(B_k(\mathbb{R}^n\setminus \bar p;\mathbb{Q})\cong \bigoplus_{l=0}^k\bigoplus_{j_1+\cdots+j_m=l} H_{*-l(n-1)}(B_{k-l}(\mathbb{R}^n);\mathbb{Q}),$$ from which it follows easily that $$\dim H_*(B_k(\mathbb{R}^n\setminus \bar p;\mathbb{Q})=\begin{cases}
\binom{m+i-1}{m-1}+\binom{m+i-2}{m-1} \quad &*=i(n-1), \, 0\leq i<k\\	
\binom{m+k-1}{m-1}\quad &*=k(n-1)\\
0\quad &\text{ else}
\end{cases}$$ (it is helpful to recall that $\binom{m+i-1}{m-1}$ is the number of ways to write $i$ as the sum of $m$ non-negative integers).

It should be clear from this example that Theorem \ref{with grading} reduces calculations to counting problem whenever $n$ is odd or the relevant compactly supported cohomology has no cup products.

\subsection{Punctured torus}
Since $H_c^{-*}(T^2\setminus \mathrm{pt};\mathbb{Q})\cong \widetilde H^{-*}(T^2;\mathbb{Q}),$ the relevant Lie algebra is isomorphic to $$\mathfrak{h}\oplus \mathbb{Q}\langle \tilde a, \tilde b, c\rangle$$ where $\mathfrak{h}=\mathbb{Q}\langle a, b, \tilde c\rangle$ as a vector space, $|a|=|b|=|\tilde c|=0$, $|\tilde a|=|\tilde b|=1$, $|c|=-1$, and the bracket is defined by the equation $$[a,b]=\tilde c.$$

The Lie homology of $\mathfrak{h}$ is calculated by the complex $$(\Lambda[ x, y, \tilde z], d(xy)=\tilde z),$$ (where for ease of notation we have set $x=\sigma a$ and so on), a basis for the homology of which is easily seen to be given by the image in homology of the set $\{1,x, y, x\tilde z,  y\tilde z, xy\tilde z\}.$ Thus we have an identification of $H_*(B_k(T^2\setminus \mathrm{pt});\mathbb{Q})$ with the weight $k$ part of $$\mathbb{Q}\langle 1,x,y,x\tilde z,y\tilde z,xy\tilde z\rangle\otimes \mathbb{Q}[\tilde x,\tilde y,z],\qquad |z|=0,\, |x|=|y|=|\tilde z|=1,\,|\tilde x|=|\tilde y|=2.$$ Counting, we find that $$\dim H_*(B_k(T^2\setminus\mathrm{pt});\mathbb{Q})=\begin{cases}
\frac{3i-1}{2}+1\quad &*=2i+1<k\\
\frac{3i}{2}+1\quad &*=2i<k\\
k+1\quad &*=k \text{ odd}\\
\frac{k}{2}+1\quad &*=k \text{ even}\\
0\quad &\text{else.}
\end{cases}$$

An amusing comparison can be seen by taking $k=2$ in the above formula, which yields $$H_*(B_2(T^2\setminus \mathrm{pt});\mathbb{Q})\cong \mathbb{Q}\oplus\mathbb{Q}^2[1]\oplus\mathbb{Q}^2[2].$$ On the other hand, from the preceding example, one calculates that $$H_*(B_2(\mathbb{R}^2\setminus \{p_1,p_2\};\mathbb{Q})\cong \mathbb{Q}\oplus\mathbb{Q}^3[1]\oplus\mathbb{Q}^3[2].$$ Thus, despite the fact that the punctured torus and the twice-punctured plane are homotopy equivalent, having $S^1\vee S^1$ as a common deformation retract, their configuration spaces are not homotopy equivalent.

\subsection{Real projective space}
Let $n$ be even, so that $\mathbb{RP}^n$ is non-orientable. Then, as a ring, $H_c^{-*}(\mathbb{RP}^n;\mathbb{Q})\cong\mathbb{Q}$, and the Lie homology of interest is $H^\Lie(\Lie(\mathbb{Q}[n-1]))\cong\mathbb{Q}\oplus \mathbb{Q}[n],$ whence, for $k>1$, $$H_*(B_k(\mathbb{RP}^n);\mathbb{Q}^w)=0.$$

As for the untwisted homology, we note that $H_c^{-*}(\mathbb{RP}^n;\mathbb{Q}^w)\cong \mathbb{Q}[-n]$ by Poincar\'{e} duality, so that the cup product map $H_c^{-*}(\mathbb{RP}^n;\mathbb{Q}^w)^{\otimes 2}\to H_c^{-*}(\mathbb{RP}^n;\mathbb{Q})$ is trivial for degree reasons. Thus $$H_c^{-*}(\mathbb{RP}^n; \Lie(\mathbb{Q}^w[n-1]))\cong \mathbb{Q}[-1] \oplus \mathbb{Q}[2n-2]$$ is abelian, so that $H_*(B_k(\mathbb{RP}^n);\mathbb{Q})$ is isomorphic to the weight $k$ part of $$\mathbb{Q}[x]\otimes \Lambda[\tilde y],\quad |x|=0,\, |\tilde y|=2n-1.$$ Hence for all $k>1$, $$H_*(B_k(\mathbb{RP}^n);\mathbb{Q})\cong\mathbb{Q}\oplus\mathbb{Q}[2n-1].$$

\subsection{Klein bottle, twisted}
Let $\mathbb{K}$ denote the Klein bottle. Then $H_c^{-*}(\mathbb{K};\mathbb{Q})\cong \mathbb{Q}\oplus\mathbb{Q}[-1]$, with the generator in degree zero acting as a unit for the multiplication. As a vector space, the Lie algebra in question is $\mathfrak{g}:=\mathbb{Q}\langle a, \tilde a, b, \tilde b\rangle$, where $|b|=0$, $|a|=|\tilde b|=1$, and $|\tilde a|=2$, and the bracket is defined by the equations $$[a,a]=\tilde a\qquad\qquad [a,b]=-\tilde b.$$ The subspace spanned by $\{b,\tilde b\}$ is an ideal realizing $\mathfrak{g}$ as an extension $$0\to \mathbb{Q}\langle b,\tilde b\rangle\to\mathfrak{g}\to \Lie(\mathbb{Q}\langle a\rangle)\to0,$$ so that we may avail ourselves of the Lyndon-Hochschild-Serre spectral sequence $$E^2_{p,q}\cong H^\Lie_p(\Lie(\mathbb{Q}\langle a\rangle);H^\Lie_q(\mathbb{Q}\langle b,\tilde b\rangle))\implies H^\Lie_{p+q}(\mathfrak{g}).$$ There are no differentials for degree reasons, and the $E^2$ page is computed as the homology of the complex $$0\to \mathbb{Q}\langle a\rangle[1]\otimes \Sym(\mathbb{Q}\langle b,\tilde b\rangle[1])\to \Sym(\mathbb{Q}\langle b,\tilde b\rangle[1])\to 0,$$ where the differential is the action of $a$. It follows that a basis for $H^\Lie(\mathfrak{g})$ is given by $\{\sigma a\otimes (\sigma\tilde b)^i,\sigma b\otimes(\sigma\tilde b)^j\mid i,j\geq0\}.$ Counting monomials of weight $k$, we find that $$H_*(B_k(\mathbb{K});\mathbb{Q}^w)\cong\begin{cases}
\mathbb{Q}[k]\oplus \mathbb{Q}[k+1]\qquad &k\text{ odd}\\
0\qquad\qquad &k\text{ even.}
\end{cases}$$

\subsection{Non-orientable surfaces}
Let $N_h=(\mathbb{RP}^2)^{\#h}$. Using the method of the previous example, one could proceed to obtain a general formula for the twisted homology of $B_k(N_h)$. Here we will determine the corresponding untwisted homology. We have $$H_c^{-*}(N_h;\mathbb{Q})\cong\mathbb{Q}\oplus\mathbb{Q}[-1]^{h-1},\qquad H_c^{-*}(N_h; \mathbb{Q}^w)\cong\mathbb{Q}[-1]^{h-1}\oplus\mathbb{Q}[-2],$$ so that there can be no cup products. Thus $H_*(B_k(N_h);\mathbb{Q})$ is the weight $k$ part of $$\mathbb{Q}[x, \tilde y_1,\ldots, \tilde y_{h-1}]\otimes \Lambda[\tilde z, w_1,\ldots, w_{h-1}],\quad |x|=0,\,|w_i|=1,\,|\tilde y_i|=2,\,|\tilde z|=3.$$ Counting inductively as in the example of punctured Euclidean space, we find that $$H_*(B_k(N_h);\mathbb{Q})\cong \bigoplus_{l=0}^k\bigoplus_{j_1+\cdots j_{h-1}=l}H_{*-l}(B_{k-l}(\mathbb{RP}^2);\mathbb{Q}),$$ from which it follows that $$\dim H_*(B_k(N_h);\mathbb{Q})=\begin{cases}
\binom{h+*-2}{h-2}+\binom{h+*-5}{h-2}\qquad &*\leq k\\
\binom{h+*-5}{h-2}\qquad &*=k+1\\
0\qquad &\text{else.}
\end{cases}$$

\subsection{Open and closed M\"{o}bius band} Let $\mathbb{M}$ denote the closed M\"{o}bius band. Then since $\mathbb{M}$ has the same compactly supported cohomology ring as the Klein bottle, our earlier calculation shows that 

$$\widetilde{H}_*(B_k(\mathbb{M},\partial \mathbb{M});\mathbb{Q}^w)\cong\begin{cases}
\mathbb{Q}[k]\oplus \mathbb{Q}[k+1]\qquad &k\text{ odd}\\
0\qquad\qquad &k\text{ even.}
\end{cases}$$ On the other hand, $H_c^{-*}(\mathbb{M};\mathbb{Q}^w)=0$ by Poincar\'{e} duality, so that $H_c^{-*}(\mathbb{M}; \Lie(\mathbb{Q}^w[1]))\cong H^{-*}(\mathbb{M};\mathbb{Q})[2]$ is abelian, and $\widetilde{H}_*(B_k(\mathbb{M},\partial \mathbb{M});\mathbb{Q})$ is the weight $k$ part of $$\mathbb{Q}[\tilde x]\otimes\Lambda[\tilde y],\quad |\tilde x|=2,\,|\tilde y|=3,$$ whence $$\widetilde{H}_*(B_k(\mathbb{M},\partial \mathbb{M});\mathbb{Q})\cong\begin{cases}
0\qquad & k\text{ odd}\\
\mathbb{Q}[k]\oplus\mathbb{Q}[k+1]\qquad & k\text{ even.}\\
\end{cases}$$

The situation with the corresponding open manifold is quite different. Since $(\mathring{\mathbb{M}})^+\cong\mathbb{RP}^2$, $H_c^{-*}(\mathring{\mathbb{M}};\mathbb{Q})=0$, so that $$H_*(B_k(\mathring{\mathbb{M}});\mathbb{Q}^w)=0$$ for all $k>1$. On the other hand, $H_c^{-*}(\mathring{\mathbb{M}};\mathbb{Q}^w)\cong\mathbb{Q}[-1]\oplus\mathbb{Q}[-2]$ by Poincar\'{e} duality, so that $H_*(B_k(\mathring{\mathbb{M}});\mathbb{Q})$ is the weight $k$ part of $$\mathbb{Q}[x]\otimes\Lambda[y],\qquad |x|=0,\,| y|=1,$$ whence $$H_*(B_k(\mathring{\mathbb{M}});\mathbb{Q})\cong \mathbb{Q}[0]\oplus\mathbb{Q}[1]$$ for all $k\geq1$.

\section{Two Formality Results} In this final section, we supply the remaining two ingredients in the proof of Theorem \ref{with grading}. Although unrelated to each other, these formality statements may be of independent interest.

\subsection{The $O(n)$-equivariant sphere}

Since the reduced homology of $S^n$ is one-dimensional, any choice of representative of a homology generator defines a quasi-isomorphism $$C_*(S^n)\simeq\mathbb{Z}\oplus\mathbb{Z}[n].$$ The goal of this section is to prove that, rationally, this equivalence can be made $O(n)$-equivariant.

\begin{theorem}\label{equivariant formality}
There is an equivalence of $O(n)$-modules $$C_*(S^n;\mathbb{Q})\simeq\mathbb{Q}\oplus\mathbb{Q}^{\det}[n].$$
\end{theorem}

The proof has three main ingredients, the first of which is rational homotopy theory. We consider the Borel construction $$\hat\xi:\xymatrix{\displaystyle ESO(n)\times_{SO(n)}S^n\ar[r] &BSO(n)}$$ where $SO(n)$ acts on $S^n\cong(\mathbb{R}^n)^+$ by extension of its canonical action on $\mathbb{R}^n$. In other words, $\hat\xi$ is the fiberwise one-point compactification of the universal oriented $n$-plane bundle $\xi$. We denote by $E(\hat\xi)$ the total space of this sphere bundle.

Sphere bundles over simply connected spaces admit particularly simple rational descriptions. According to \cite[15(a),\,15(b)]{FHT}, we have the following commutative diagram, whose terms we will explain presently:

$$\xymatrix{
(\Sym(W_n),d_1)\ar[r]^-\sim&A_{PL}(S^n)\ar[r]^-\sim&C^{-*}(S^n;\mathbb{Q})\\
(S\otimes\Sym(W_n),d_1+d_2)\ar[u]\ar[r]^-\sim\ar[u]&A_{PL}(E(\hat\xi))\ar[r]^-\sim\ar[u]&C^{-*}(E(\hat\xi);\mathbb{Q})\ar[u]\\
S\ar[r]^-\sim\ar[u]&A_{PL}(BSO(n))\ar[u]\ar[r]^-\sim&C^{-*}(BSO(n);\mathbb{Q})\ar[u]
}$$

In this diagram, \begin{enumerate}
\item $S:=H^{-*}(BSO(n);\mathbb{Q})$ is a polynomial algebra;
\item $A_{PL}$ denotes the functor of PL de Rham forms;
\item the horizontal arrows in the righthand column are components of the natural quasi-isomorphism $\oint: A_{PL}\to C^{-*}$ given by integrating forms over simplices;
\item each term appearing in the leftmost column is a Sullivan model for the corresponding space;
\item $W_n$ denotes the graded vector space $$W_n=\begin{cases}
\mathbb{Q}\langle x_{-n}\rangle\quad &n\text{ odd}\\
\mathbb{Q}\langle x_{-n}, y_{-2n+1}\rangle\quad &n\text{ even}
\end{cases}$$ the differential $d_1$ is defined by the equation $d_1(y)=x^2$, and the differential $d_2$ is specified by its value on $y$, which is an element of $P$ determined by the bundle $\hat\xi$.
\end{enumerate} We direct the reader to \cite{Borel} for more on (i), and to \cite[10(c),\,10(e),\,12,\,15(b)]{FHT}, respectively, for more on (ii)-(v). The reader is advised that, although we have maintained our convention of homological grading, the prevailing convention in rational homotopy theory is cohomological.

The second ingredient is the theory of $A_\infty$-algebras and their modules, for which we refer the reader to \cite{Keller}. The relevance here is that, according to \cite[3.1]{BG}, the integration map $\oint$ extends to a map of $A_\infty$-algebras (referred to in \cite{BG} as ``strongly homotopic differential algebras''), so that $C^{-*}(E(\hat\xi);\mathbb{Q})$ becomes an $A_\infty$-$S$-module via the bottom composite in the above diagram.

\begin{proposition}\label{dual side}
There is a quasi-isomorphism of $A_\infty$-$S$-modules $$\xymatrix{S\oplus S[-n]\ar[r]^-\sim &C^{-*}(E(\hat\xi);\mathbb{Q}).}$$
\end{proposition}
\begin{proof} 
The fiberwise basepoint furnishes $\hat\xi$ with a section, and the Gysin sequence now implies that the top map in the commuting diagram $$\xymatrix{
S\oplus S[-n]\ar[r]&(S\otimes\Sym(W_n),d_1+d_2)\\
S\ar[u]\ar@{=}[r]&S\ar[u]
}$$ is a quasi-isomorphism. Combining this diagram with the previous yields the result.
\end{proof}

The third ingredient is the Koszul duality between modules for the symmetric algebra $S$ and modules for the exterior algebra $\Lambda$ on the same generators with degrees shifted by 1. According to \cite{Borel}, there is a Hopf algebra isomorphism $\Lambda\cong H_*(SO(n);\mathbb{Q})$, where the latter carries the Pontryagin product induced by the group structure of $SO(n)$. Koszul duality is the algebraic avatar of the correspondence between $SO(n)$-spaces and spaces fibered over $BSO(n)$ witnessed by the Borel construction. There are many variations on this theme; the relevant facts for our purposes are the following, which are extracted from \cite[1.2,\,3.1]{Franz}; see also \cite{GKM}. Our notation differs slightly from that in \cite{Franz}, and we maintain the terminology of $A_\infty$-modules rather than ``weak modules.''

\begin{theorem}[(Franz, Goresky-Kottwitz-MacPherson)]\label{KD}
There is a functor $\mathbf{h}$ from $A_\infty$-S-modules to $A_\infty$-$\Lambda$-modules with the following properties. \begin{enumerate}
\item Let $\pi:X\to BSO(n)$ be a space over $BSO(n)$. Then $\mathbf{h}(C^{-*}(X))$ and $C^{-*}(\hofiber(\pi))$ are connected by a zig-zag of natural quasi-isomorphisms of $A_\infty$-$\Lambda$-modules.
\item Let $V$ be a graded vector space. Then $\mathbf{h}(S\otimes V)\cong V$, where $V$ is regarded as a trivial $A_\infty$-$\Lambda$-module.
\end{enumerate}
\end{theorem}

\begin{proposition}\label{special orthogonal}
There is an equivalence of $SO(n)$-modules $$C_*(S^n;\mathbb{Q})\simeq \mathbb{Q}\oplus\mathbb{Q}[n],$$ where the latter is regarded as a trivial $SO(n)$-module.
\end{proposition}
\begin{proof}
Both of the $SO(n)$-modules in question are dualizable objects of $\Ch_\mathbb{Q}$, so it suffices to exhibit an $SO(n)$-equivalence $C^{-*}(SO(n);\mathbb{Q})\simeq \mathbb{Q}\oplus\mathbb{Q}[-n]$ between the duals. By \cite[4.3.3.17]{HA}, the homotopy category of the $\infty$-category of $SO(n)$-modules coincides with the homotopy category obtained from the model category of $C_*(SO(n);\mathbb{Q})$-modules equipped with the usual model structure on modules over a differential graded algebra. By \cite[4.3]{Keller}, this homotopy category in turn coincides with the full subcategory of the homotopy category of $A_\infty$-$C_*(SO(n);\mathbb{Q})$-modules spanned by the ``homologically unital modules,'' so that, since the modules in question are homologically unital, it will suffice to produce to an isomorphism in the homotopy category of $A_\infty$-modules. By \cite[6.2]{Keller}, it suffices to produce an isomorphism in the homotopy category of $A_\infty$-$\Lambda$-modules after restricting along the $A_\infty$-quasi-isomorphism $\Lambda\to C_*(SO(n);\mathbb{Q})$ of \cite{Franz}. For this, we apply the Koszul duality of Theorem \ref{KD} to the $A_\infty$-quasi-isomorphism of Proposition \ref{dual side}, yielding the zig-zag of $A_\infty$-quasi-isomorphisms $$\xymatrix{\mathbb{Q}\oplus\mathbb{Q}[-n]\simeq \mathbf{h}(S\oplus S[-n])\ar[r]&\mathbf{h}(C^{-*}(E(\hat\xi);\mathbb{Q}))\simeq C^{-*}(S^n;\mathbb{Q}).}$$
\end{proof}

\begin{proof}[Proof of Theorem \ref{equivariant formality}]

We explain the following diagram of $O(n)$-modules: $$\xymatrix{
\mathbb{Q}\oplus\mathbb{Q}^{\det}[n]\ar[r]&\mathbb{Q}[C_2]\oplus\mathbb{Q}[C_2][n]\simeq \mathbb{Q}[C_2]\otimes C_*(S^n;\mathbb{Q})\ar[r]& C_*(S^n;\mathbb{Q}).
}$$ 

\begin{enumerate}
\item Let $e$ and $\sigma$ denote the basis elements of $\mathbb{Q}[C_2]$ corresponding to the identity and generator, respectively. The lefthand map sends $1\in\mathbb{Q}$ to $\frac{e+\sigma}{2}$ and $1\in\mathbb{Q}^{\det}$ to $\frac{e-\sigma}{2}$. This is a map of $C_2$-modules and therefore of $O(n)$-modules, since $O(n)$ acts on both domain and codomain by restriction along the determinant.
\item Fixing a choice of isomorphism $O(n)\cong C_2\ltimes SO(n)$, we obtain an isomorphism $C_*(O(n);\mathbb{Q})\cong \mathbb{Q}[C_2]\otimes C_*(SO(n);\mathbb{Q})$ of $O(n)$-modules. The middle equivalence is now obtained by applying the functor of induction from $SO(n)$ to $O(n)$ to the equivalence of Proposition \ref{special orthogonal}.
\item The righthand arrow is the counit of the induction-restriction adjunction.
\end{enumerate}
Applying homology yields an isomorphism, completing the proof.
\end{proof}

\subsection{Two-step nilpotent Lie algebras}

In this section, we prove that the Lie algebras of interest to us are formal.

\begin{proposition}\label{Lie formality}
Let $K$ be either $\mathbb{Q}$ or $\mathbb{Q}^{\sgn}$. For any $r\in \mathbb{Z}$ and any manifold $M$, the Lie algebra $\Map^{C_2}(\widetilde M^+,\Lie(K[r]))$ is formal.
\end{proposition}

The proof will rely on the following technical result.

\begin{proposition}\label{technical lemma}
Let $$0\to\mathfrak{h}\to\mathfrak{e}\to\mathfrak{g}\to0$$ be an exact sequence of Lie algebras in $\Ch_\mathbb{Q}$ with $\mathfrak{g}$ and $\mathfrak{h}$ abelian. Assume that $\mathfrak{g}$ acts trivially on $\mathfrak{h}$ and that the underlying sequence of chain complexes splits. Then $\mathfrak{e}$ is formal.
\end{proposition}
\begin{proof}
The hypotheses imply that the bracket on $\mathfrak{e}\cong\mathfrak{g}\oplus\mathfrak{h}$ is given by $$[(g_1,h_1),(g_2,h_2)]=f(g_1,g_2)$$ for some (not uniquely defined) map $f:\Sym^2(\mathfrak{g}[1])[-2]\to\mathfrak{h}$, and the bracket on $H(\mathfrak{e})\cong H(\mathfrak{g})\oplus H(\mathfrak{h})$ is determined in the same way by $f_*$.

Choose quasi-isomorphisms $\varphi:\mathfrak{g}\to H(\mathfrak{g})$ and $\psi:\mathfrak{h}\to H(\mathfrak{h})$. Without loss of generality, we may assume that both maps induce the identity on homology. Let $\bar\psi$ be a quasi-inverse to $\psi$. Then $(\bar\psi\circ f_*\circ\varphi^{\wedge2})_*=f_*$, so $$\bar\psi\circ f_*\circ\varphi^{\wedge 2}-f=d_{\mathfrak{h}}G+Gd_{\Sym}$$ for some homotopy operator $G:\Sym^2(\mathfrak{g}[1])[-2]\to \mathfrak{h}[-1]$. 

Now, since $\mathfrak{g}$ is abelian and acts trivially, this equation may be written as $$D(G)=\bar\psi\circ f_*\circ\varphi^{\wedge 2}-f,$$ where $D$ denotes the differential in the Chevalley-Eilenberg cochain complex computing $H^*_\Lie(\mathfrak{g},\mathfrak{h})$. Since extensions of $\mathfrak{g}$ by the module $\mathfrak{h}$ are classified by $H^2_\Lie(\mathfrak{g},\mathfrak{h})$, it follows that $f$ and $\bar\psi\circ f_*\circ\varphi^{\wedge2}$ determine isomorphic extensions, so that we may take $f=\bar\psi\circ f_*\circ\varphi^{\wedge 2}$ after choosing a different splitting. But then $\psi\circ f=f_*\circ \varphi^{\wedge 2}$, so that the composite $$\begin{CD}
\mathfrak{e}@>\cong>>\mathfrak{g}\oplus\mathfrak{h}@>(\varphi,\psi)>>H(\mathfrak{g})\oplus H(\mathfrak{h})@>\cong>>H(\mathfrak{e})
\end{CD}$$ is a map of Lie algebras. Since it is also a quasi-isomorphism of chain complexes, the proof is complete.
\end{proof}

\begin{proof}[Proof of Proposition \ref{Lie formality}]
The exact sequence $$0\to \mathfrak{h}\to \Lie(K[r])\to K[r]\to0$$ satisfies the hypotheses of Proposition \ref{technical lemma}, where $$\mathfrak{h}=\begin{cases}
K^{\otimes 2}[2r]&\quad r\text{ odd}\\
0&\quad r\text{ even.}
\end{cases}$$ By Proposition \ref{cotensor identification}, we have $$\textstyle\Map^{C_2}(\widetilde M^+,\Lie(K[r])\simeq (A_{PL}(\widetilde M^+)\otimes\Lie(K[r]))^{C_2}.$$ Since the operations of tensoring with the commutative algebra $A_{PL}(\widetilde M^+)$ and taking $C_2$ fixed points preserve the hypotheses of Proposition \ref{technical lemma} the claim follows. 
\end{proof}

\begin{remark}
Proposition \ref{Lie formality} asserts that $A_{PL}(M)\otimes\Lie(\mathbb{Q}[r])$ is formal whenever $M$ is compact and orientable. When $r$ is odd, this fact may be suprising at first glance, since $M$ is not assumed to be formal. 

A conceptual understanding of this phenomenon is afforded by the homotopy transfer theorem (see \cite[10.3]{LV}, for example). Indeed, let $A$ be any nonunital differential graded commutative algebra and $\mathfrak{g}$ any two-step nilpotent graded Lie algebra. Fixing an additive homotopy equivalence between $A$ and $H(A)$, we obtain a transferred $L_\infty$-algebra structure on $H(A)\otimes \mathfrak{g}$. The higher brackets of the transferred structure combine information about the Massey products of $A$ and the Lie bracket of $\mathfrak{g}$.

In our case, using the fact that $\mathfrak{g}$ has no nontrivial iterated brackets, the explicit formulas of the homotopy transfer theorem show that these higher brackets all vanish, which implies that $A\otimes\mathfrak{g}$ is formal. In other words, although the Massey products in $H(A)$ may be nontrivial, they are damped out by the nilpotence of $\mathfrak{g}$.
\end{remark}

%%%%%%%%%%%%%%%%%%%%   End of main body of article
%
%                             References
%
%   BiBTeX users uncomment the following line:
%
%\bibliographystyle{gtart}
%

\end{document}